\documentclass[11pt]{article}
\usepackage{caption2}
\usepackage{amssymb}
\usepackage{color}
\usepackage{amsmath,amsthm}
\usepackage{graphicx}
\usepackage{empheq}
\usepackage{indentfirst}
\usepackage[hypertex]{hyperref}
 \newtheorem{thm}{Theorem}[section]
 
 \newtheorem{lem}{Lemma}[section]
 
 \theoremstyle{definition}
 \newtheorem{defn}{Definition}[section]
 \newtheorem{rem}{Remark}[section]
 \numberwithin{equation}{section}
\def\f{\frac}

\def\vep{\varepsilon}

\def\wt #1{\widetilde{#1}}
\def\ov #1{\overline{#1}}
\def\i1n{i=1,\cdots,n}
\def\j1n{j=1,\cdots,n}
\def\ij1n{i,j=1,\cdots,n}

\def\R{\mathbb R}

\newcommand{\be}{\begin{equation}}
\newcommand{\ee}{\end{equation}}
\newcommand{\beq}{\begin{equation*}}
\newcommand{\eeq}{\end{equation*}}

\DeclareMathOperator*{\esssup}{ess\,sup}

\title{Analysis and control of a scalar conservation law \\
   modeling a highly re-entrant manufacturing system}%
\author{
Peipei SHANG\thanks{INRIA Paris-Rocquencourt Centre. Universit\'{e}
Pierre et Marie Curie-Paris 6, UMR 7598 Laboratoire Jacques-Louis
Lions, 75005 Paris, France. E-mail: {\tt Peipei.Shang@inria.fr}. PS
was supported by the large scale INRIA project REGATE (REgulation of
the GonAdoTropE axis).}
\ and Zhiqiang WANG\thanks{School of Mathematical Sciences, Fudan
University, Shanghai 200433, China. Universit\'{e} Pierre et Marie
Curie-Paris 6, UMR 7598 Laboratoire Jacques-Louis Lions, 75005
Paris, France. E-mail: {\tt wzq@fudan.edu.cn}. ZW was partially
supported by the Natural Science Foundation of China grant 10701028
and Fondation Sciences Math\'{e}matiques de Paris. } 
}
\date{March 15, 2010}
\begin{document}
%
%

\maketitle
\begin{abstract}
In this paper, we study a scalar conservation law that models a
highly re-entrant manufacturing system as encountered in
semi-conductor production. As a generalization of \cite{CKWang}, the
velocity function possesses both the local and nonlocal character.
We prove the existence and uniqueness of the weak solution to the
Cauchy problem with initial and boundary data in $L^{\infty}$. We
also obtain the stability (continuous dependence) of both the
solution and the out-flux with respect to the initial and boundary
data. Finally, we prove the existence of an optimal control that
minimizes, in the $L^p$-sense with $p\in [1,\infty)$, the difference
between the actual out-flux and a forecast demand over a fixed time period. \\
\end{abstract}
{\bf Keywords:}\quad Conservation law, nonlocal velocity,
stability, optimal control, re-entrant manufacturing system.\\
{\bf 2000 MR Subject Classification:}\quad
         35L65, 
         49J20, 
         93C20. 
\section{Introduction and main results}

In this paper, we study the scalar conservation law
\be \label{eq} \rho_t(t,x)+(\rho(t,x)\lambda(x,W(t)))_x=0, \quad
t\geq 0, 0\leq x\leq 1, \ee
where
\beq W(t)=\int_0^1\rho(t,x)dx. \eeq
We assume that the velocity function $\lambda >0$ is continuous
differentiable, i.e., $\lambda \in C^1([0,1] \times [0,\infty))$, in
the whole paper. For instance, we recall that the special case of
\beq \lambda(x,W)=\frac{1}{1+W} \eeq
was used in \cite{Armbruster06, LaMarca}.

This work is motivated by problems arising in the control of
semiconductor manufacturing systems which are characterized by their
highly re-entrant feature. This character is, in particular,
described in terms of the velocity function $\lambda$ in the model:
it is a function of the total mass $W(t)$ (the integral of the
density $\rho$). As a generalization of \cite{CKWang} (in which
$\lambda=\lambda(W(t))$), here we assume that the velocity $\lambda$
varies also with respect to the local position $x$, as can be
naturally encountered in practice. These phenomena also appear in
some biologic models (modeling the development of ovarian follicles,
see \cite{Clement05, Clement07}) and pedestrian flow models (see
\cite{Coscia, Colombo09, CHerty09}).

In the manufacturing system, with a given initial data
 \be \label{eq-IC} \rho(0,x)=\rho_0(x),  \quad 0\leq x \leq 1, \ee
the natural control input is the in-flux, which suggests the
boundary condition
 \be \label{eq-BC}\rho(t,0)\lambda(0,W(t)) = u(t), \quad t\geq 0. \ee
Motivated by applications, one natural control problem is related to
the \emph{Demand Tracking Problem (DTP)}. The objective of
\emph{DTP} is to minimize the difference between the actual out-flux
$y(t)=\rho(t,1)\lambda(1,W(t))$ and a given demand forecast $y_d(t)$
over a fixed time period. An alternative control problem is
\emph{Backlog Problem (BP)}. The objective of \emph{BP} is to
minimize the difference between the number of the total products
that have left the factory and the number of the total demanded
products over a fixed time period. The backlog of a production
system at a given time $t$ is defined as
  \beq \label{opt-backlog}
  \beta (t)=\int_0^t \rho(s,1)\lambda(1,W(s))ds -\int_0^t y_d(s) ds.
   \eeq
The backlog $\beta(t)$ can be negative or positive, with a positive
backlog corresponding to overproduction and a negative backlog
corresponding to a shortage.

Partial differential equation models for such manufacturing systems
are motivated by the very high volume (number of parts manufactured
per unit time) and the very large number of consecutive production
steps. They are popular due to their superior analytic properties
and the availability of efficient numerical tools for simulation.
For more detailed discussions, see e.g. \cite{Armbruster06SIAM,
Armbruster06, Herty07, LaMarca}. In many aspects these models are
quite similar to those of traffic flows \cite{Coclite05} and
pedestrian flows \cite{Coscia, Colombo09, CHerty09}.

The hyperbolic conservation laws and related control problems have
been widely studied for a long time. The fundamental problems
include the existence, uniqueness, regularity and continuous
dependence of solutions, controllability, asymptotic stabilization,
existence and uniqueness of optimal controls. For the well-posedness
problems, we refer to the works \cite{BBressan, BressanBook,
LeFlochBook, LiuYang} (and the references therein) in the content of
weak solutions to systems (including scalar case) in conservation
laws, and \cite{LiBook94, LiYuBook} in the content of classical
solutions to general quasi-linear hyperbolic systems. For the
controllability of linear hyperbolic systems, one can see the
important survey \cite{Russell}. The controllability of nonlinear
hyperbolic equations (or systems) are studied in \cite{Coron,
CGWang, Glass, Gugat, Horsin, LiBook09, LiRao}, while the attainable
set and asymptotic stabilization of conservation laws can be found
in \cite{Ancona07, Ancona98}. In particular, \cite{CoronBook}
provides a comprehensive survey of controllability and stabilization
in partial differential equations that also includes nonlinear
conservation laws.

We prove the existence, uniqueness and regularity of the weak
solution to Cauchy problem \eqref{eq}, \eqref{eq-IC} and
\eqref{eq-BC} with initial and boundary data in $L^{\infty}$. The
main approach is the characteristic method. We point out here that
in the previous paper \cite{CKWang}, the authors obtained the
well-posedness for  $L^p\ (1\leq p<\infty)$ data. The $L^{\infty}$
assumption in this paper is due to the fact that the velocity
function $\lambda$ depends on the space variable $x$. Using the
implicit expression of the solution in terms of the characteristics,
we also prove the stability (continuous dependence) of both the
solution and the out-flux with respect to the initial and boundary
data. The stability property guarantees that a small perturbation to
the initial and (or) boundary data produces also only a small
perturbation to the solution and the out-flux.

The optimal control problem that we study in this paper is related
to the \emph{Demand Tracking Problem}. This problem is motivated by
\cite{CKWang} and originally inspired by \cite{LaMarca}. The
objective is to minimize the $L^p$-norm with $p\in [1,\infty)$ of
the difference between the actual out-flux $y(t)=\rho(t,1)
\lambda(1,W(t))$ and a given demand forecast $y_d(t)$ over a fixed
time period. With the help of the implicit expression of the weak
solution and by compactness arguments, we prove the existence of
solutions to this optimal control problem.

The main difficulty of this paper comes from the nonlocal velocity
in the model. A related manuscript \cite{CHerty09}, which is also
motivated in part by \cite{Armbruster06, LaMarca}, addressed
well-posedness for systems of hyperbolic conservation laws with a
nonlocal velocity in $\R^n$. The authors studied the Cauchy problem
in the whole space $\R^n$ without considering any boundary
conditions and they also gave a necessary condition for the possible
optimal controls. However, the method of proof and even the
definition of solutions are different from this paper. Another
scalar conservation law with nonlocal velocity is to model
sedimentation of particles in a dilute fluid suspension, see
\cite{Zumbrun} for the well-posedness of the Cauchy problem. In this
model, the nonlocal velocity is due to a convolution of the unknown
function with a symmetric smoothing kernel. There are also some
other one-dimensional models with nonlocal velocity, either in
divergence form or not,  which are related to the 3D Navier-Stokes
equations or the Euler equations in the vorticity formulation.
Nevertheless, the nonlocal character in these models comes from a
singular integral of the unknown function (see \cite{Dong} and the
references therein, especially \cite{CLMajda}).

The organization of this paper is as follows: First in Section 2
some basic notations and assumptions are given. Next in Section 3 we
prove the existence and uniqueness of the weak solution to Cauchy
problem \eqref{eq}, \eqref{eq-IC} and \eqref{eq-BC} with the initial
data $\rho_0\in L^{\infty}(0,1)$ and boundary data $u\in
L^{\infty}(0,T)$. Some remarks on the regularity of the weak
solution to the Cauchy problem are also given in Section 3. In
Section 4 we establish the stability of the weak solution and the
out-flux with respect to the initial and boundary data. Then in
Section 5, we prove the existence of the solution to the optimal
control problem of minimizing the $L^p$-norm of the difference
between the actual and any desired (forecast) out-flux. Finally in
the appendix, we give two basic lemmas and the proofs of Lemmas
\ref{rho-sol}-\ref{uni-sol} that are used in Section 3.

\section{Preliminaries}
First we introduce some notations which will be used in the whole
paper:
 \begin{align*}
 &  L^{\infty}_+(0,1):=\{f\in L^{\infty}(0,1)\colon
   \text{$f$ is nonnegative almost everywhere}\},
   \\
 &  L^{\infty}_+(0,T):=\{f\in L^{\infty}(0,T)\colon
      \text{$f$ is nonnegative almost everywhere}\},
   \\
 & \|\rho_0\|_{L^{\infty}}:= \|\rho_0\|_{L^{\infty}(0,1)}:= \esssup_{0\leq x\leq 1}
    |\rho_0(x)|,
   \\
 &   \|u\|_{L^{\infty}}:=\|u\|_{L^{\infty}(0,T)}:= \esssup_{0\leq x\leq T} |u(t)|
\end{align*}
and
 \begin{align}
   \label{def-M}
 & M:=\|u\|_{L^1(0,T)}+\|\rho_0\|_{L^1(0,1)},
   \\\label{ov-lam}
 & \ov \lambda(M):= \inf_{(x,W)\in [0,1]\times [0,M]}
   \lambda(x,W)>0,
    \\ \label{lam-bound}
 & \|\lambda\|_{C^0}:= \|\lambda\|_{C^0([0,1]\times [0,M])}:=\sup_{(x,W)\in [0,1]\times [0,M]}
   |\lambda(x,W)|,
    \\ \label{lam-x-bound}
 & \|\lambda_x\|_{C^0}:=\|\lambda_x\|_{C^0([0,1]\times [0,M])} := \sup_{(x,W)\in [0,1]\times [0,M]}
   |\lambda_x(x,W)|,
    \\ \label{lam-w-bound}
 & \|\lambda_W\|_{C^0}:= \|\lambda_W\|_{C^0([0,1]\times [0,M])} :=\sup_{(x,W)\in [0,1]\times [0,M]}
   |\lambda_W(x,W)|.
 \end{align}

We also define the characteristic curve $\xi=\xi(s;t,x)$, which
passes through the point $(t,x)$, by the solution to the ordinary
differential equation
 \be \label{xi} \f{d\xi}{ds}=\lambda(\xi(s),W(s)), \quad
 \xi(t)=x, \ee
where $W$ is a continuous function. The existence and uniqueness of
the solution to \eqref{xi} is guaranteed by the assumption that
$\lambda\in C^1([0,1]\times [0,\infty))$ with $|W(s)|\leq M$ for all
$s$. The characteristic curve $\xi$ is frequently used afterward and
it is precisely illustrated in different situations.

\section{Well-posedness of Cauchy Problem with $L^{\infty}$ data}

First we recall, from \cite[Section 2.1]{CoronBook}, the usual
definition of a weak solution to Cauchy problem \eqref{eq},
\eqref{eq-IC} and \eqref{eq-BC}.
\begin{defn}
\label{weaksol} Let $T>0$, $\rho_0\in L^{\infty}(0,1)$ and $u\in
L^{\infty}(0,T)$ be given. A weak solution of Cauchy problem
\eqref{eq}, \eqref{eq-IC} and \eqref{eq-BC} is a function $\rho\in
C^0([0,T];L^1(0,1)) \cap L^{\infty}((0,T)\times (0,1))$ such that,
for every $\tau\in[0,T]$ and every $\varphi\in
C^1([0,\tau]\times[0,1])$ with
 \beq
 \varphi(\tau,x)= 0,\ \forall x\in[0,1]\quad \text{and}\quad
 \varphi(t,1)= 0,\ \forall t\in[0,\tau],
 \eeq
one has
 \begin{multline*}
 \int_0^{\tau} \int_0^1 \rho(t,x)(\varphi_t(t,x)
 +\lambda(x,W(t))\varphi_x(t,x)) dx dt
   \\
  +\int_0^{\tau} u(t)\varphi(t,0)dt
  +\int_0^1 \rho_0(x)\varphi(0,x)dx=0.
 \end{multline*}
\end{defn}

\begin{thm}\label{thm-sol}
Let $T>0$, $\rho_0\in L^{\infty}_+(0,1)$ and $u\in
L^{\infty}_+(0,T)$ be given, then Cauchy problem \eqref{eq},
\eqref{eq-IC} and \eqref{eq-BC} admits a unique weak solution
$\rho\in C^0([0,T];L^1(0,1)) \cap L^{\infty}((0,T)\times (0,1))$,
which is nonnegative almost everywhere in $[0,T]\times [0,1]$.
Moreover, the weak solution $\rho$ even belongs to
$C^0([0,T];L^p(0,1))$ for all $p\in [1,\infty)$.
\end{thm}

\begin{proof} Our proof is partly inspired from \cite{CGWang}.
We first prove the existence of weak solution for small time: there
exists a small $\delta\in (0,T]$ such that Cauchy problem
\eqref{eq}, \eqref{eq-IC} and \eqref{eq-BC} has a weak solution
$\rho\in C^0([0,\delta];L^1(0,1)) \cap L^{\infty}((0,\delta)\times
(0,1))$. The idea is first to prove that the total mass $W(t)$
exists as a fixed point of a map $W \mapsto F(W)$, and then to
construct a (unique) solution to the Cauchy problem.

Let
 \be\label{Omega}
 \Omega_{\delta,M}
  :=\Big\{W\in C^0([0,\delta]) \colon
  \|W\|_{C^0[0,\delta]}:=\sup_{0\leq t\leq \delta} |W(t)|\leq M
  \Big\},
 \ee
where the constant $M$ is given by \eqref{def-M}.

For any small $\delta>0$, we define a map $F:\Omega_{\delta,M}
\rightarrow C^0([0,\delta])$, $W\mapsto F(W)$, as
 \be \label{FW}
  F(W)(t):=\int_0^t u(\alpha)d\alpha
   +\int_0^{1-\int_0^t\lambda(\xi_1(\theta),W(\theta))d\theta} \rho_0(\beta)d\beta,
   \quad  \forall t \in [0,\delta],
 \ee
where $\xi_1$ (see Fig \ref{Fig1} or Fig \ref{Fig2}) represents the
characteristic curve passing through $(t,1)$ which is defined by
 \be \label{xi1} \f{d\xi_1}{ds}=\lambda(\xi_1(s),W(s)), \quad
 \xi_1(t)=1.\ee
Here we remark that the formulation of $F(W)$ is induced by solving
the corresponding \emph{linear Cauchy problem} \eqref{eq},
\eqref{eq-IC} and \eqref{eq-BC} in which $W(\cdot)\in
\Omega_{\delta,M}$ is known. It is obvious that $F$ maps into
$\Omega_{\delta,M}$ itself if
 \beq   0<\delta \leq \min \Big \{ \frac{1}{\|\lambda\|_{C^0}}, T \Big\}.
 \eeq

Now we prove that, if $\delta$ is small enough, $F$ is a contraction
mapping on $\Omega_{\delta,M}$ with respect to the $C^0$ norm. Let
$W,\ov W \in \Omega_{\delta,M}$ and for any fixed $t\in [0,T]$, we
define $\ov \xi_1$ by
 \beq \f{d\ov\xi_1}{ds}=\lambda(\ov\xi_1(s),\ov W(s)), \quad
 \ov\xi_1(t)=1.\eeq
Then, we have for every $t\in [0,\delta]$ that
 \begin{align*}
  &\quad |F(\ov W)(t)-F(W)(t)| =
     \Big| \int_{1-\int_0^t\lambda(\xi_1(\theta),W(\theta))d\theta}
     ^{1-\int_0^t\lambda(\ov \xi_1(\theta),\ov W(\theta))d\theta}
     \rho_0(\beta)d\beta \Big|
   \\
  &\leq \|\rho_0\|_{L^{\infty}} \cdot
     \Big| \int_0^t (\lambda(\ov \xi_1(\theta),\ov W(\theta))
       -\lambda(\xi_1(\theta),W(\theta))) d\theta  \Big|
   \\
  & \leq t \|\rho_0\|_{L^{\infty}} \cdot
     \Big( \|\lambda_x\|_{C^0} \|\ov \xi_1-\xi_1\|_{C^0([0,t])}
      + \|\lambda_W\|_{C^0} \|\ov W-W\|_{C^0([0,\delta])} \Big).
 \end{align*}
 By the definitions of $\xi_1$ and $\ov \xi_1$, we obtain
 \begin{align*}
   & \quad  \|\ov \xi_1-\xi_1\|_{C^0([0,t])}
    = \sup_{0\leq \theta \leq t} |\ov \xi_1(\theta)-\xi_1(\theta)|
   \nonumber    \\
  &  = \sup_{0\leq \theta \leq t} \Big| \int_{\theta}^t (\lambda (\ov \xi_1(\sigma),\ov W(\sigma))
      -\lambda (\xi_1(\sigma), W(\sigma))) d\sigma \Big |
    \nonumber  \\
   & \leq t \|\lambda_x\|_{C^0} \|\ov \xi_1-\xi_1\|_{C^0([0,t])}
       +  t \|\lambda_W\|_{C^0} \|\ov W- W\|_{C^0([0,\delta])},
 \end{align*}
thus
 \be \label{ov xi1-xi1} \|\ov \xi_1-\xi_1\|_{C^0([0,t])}
   \leq  \f {t \|\lambda_W\|_{C^0}}{1-t \|\lambda_x\|_{C^0}}
   \cdot \|\ov W-W\|_{C^0([0,\delta])}.  \ee
 Therefore,
 \beq
  |F(\ov W)(t)-F(W)(t)| \leq
  \f {t \|\rho_0\|_{L^{\infty}} \|\lambda_W\|_{C^0}}{1-t \|\lambda_x\|_{C^0} }
   \cdot \|\ov W-W\|_{C^0([0,\delta])}, \quad \forall t\in [0,\delta].
 \eeq
Let $\delta$ be such that
 \be \label{delta}
  0< \delta \leq \min \Big\{\f{1}{2 \|\lambda_x\|_{C^0} },
   \f{1}{4 \|\rho_0\|_{L^{\infty}} \|\lambda_W\|_{C^0} },
    \frac{1}{\|\lambda\|_{C^0}}, T \Big\},   \ee
then
 \beq
  \|F(\ov W)-F(W)\|_{C^0([0,\delta])}
  \leq \f 12 \|\ov W-W\|_{C^0([0,\delta])}.
 \eeq
By means of the contraction mapping principle, there exists a unique
fixed point $W=F(W)$ in $\Omega_{\delta,M}$:
  \beq W(t)=F(W)(t)=\int_0^t u(\alpha)d\alpha
   +\int_0^{1-\int_0^t\lambda(\xi_1(\theta),W(\theta))d\theta} \rho_0(\beta)d\beta,
   \quad  \forall t \in [0,\delta].
 \eeq
Moreover, $W$ is Lipschitz continuous:
 \beq  W(t)=W(0)+\int_0^t W'(s)ds,  \eeq
with
   \beq
     W'(t)=u(t)-\lambda(\xi_1(t),W(t)) \rho_0(1-\int_0^t \lambda(\xi_1(\theta),W(\theta))d\theta),
     \quad  t\in [0, \delta],
  \eeq
and thus
 \be \label{w't-bound-loc}
     \|W'\|_{L^{\infty}(0,\delta)}
   \leq  \|u\|_{L^{\infty}}
         + \|\lambda\|_{C^0} \|\rho_0\|_{L^{\infty}}.  \ee

Now we define the characteristic curve $\xi_2$ which passes through
the origin (see Fig \ref{Fig1} and Fig \ref{Fig2}) by
 \be \label{xi2} \f{d\xi_2}{ds}=\lambda(\xi_2(s),W(s)),
   \quad \xi_2(0)=0.\ee
And then for any fixed $t\in [0,\delta]$, we define the
characteristic curves $\xi_3, \xi_4$ (see Fig \ref{Fig1} and Fig
\ref{Fig2}) which pass through $(t,x)$ by
 \begin{align}\label{xi3}
   &\f{d\xi_3}{ds}=\lambda(\xi_3(s),W(s)), \quad \xi_3(t)=x, \quad \text{for}\  x \in
   [0,\xi_2(t)].
    \\\label{xi4}
   &\f{d\xi_4}{ds}=\lambda(\xi_4(s),W(s)), \quad \xi_4(t)=x, \quad \text{for}\  x \in
   [\xi_2(t),1].
 \end{align}
 From the uniqueness of the solution to the ordinary differential
 equation, we know that there exist $\alpha\in [0,\delta]$ and $\beta\in
 [0,1]$ such that
 \be \label{alp-bet} \xi_3(\alpha)=0
 \quad \text{and } \quad \xi_4(0)=\beta.
 \ee
Now we define a function $\rho$ by
 \be \label{rho-loc}
 \rho(t,x):=
   \begin{cases}
    \displaystyle \f{u(\alpha)}{\lambda(0,W(\alpha))}
     \, e^{-\int_{\alpha}^t \lambda_x(\xi_3(\theta),W(\theta)) d\theta},\quad
   & 0\leq x \leq \xi_2(t) ,0\leq t\leq \delta,
   \\
   \rho_0(\beta)\,  e^{-\int_0^t \lambda_x(\xi_4(\theta),W(\theta)) d\theta },\quad
   & 0\leq \xi_2(t)\leq x, 0\leq t\leq \delta,
  \end{cases}
 \ee
which is obviously nonnegative almost everywhere in $(0,\delta)
\times (0,1)$. Using the following two lemmas, we can prove that
$\rho$ defined by \eqref{rho-loc} is the unique weak solution to the
Cauchy problem \eqref{eq}, \eqref{eq-IC} and \eqref{eq-BC}.

\begin{lem} \label{rho-sol} The function $\rho$ defined by \eqref{rho-loc}
is a weak solution to Cauchy problem \eqref{eq}, \eqref{eq-IC} and
\eqref{eq-BC}. Moreover, the weak solution $\rho$ even belongs to
$C^0([0,\delta];L^p(0,1))$ for all $p\in [1,\infty)$ and the
following two estimates hold for all $t\in [0,\delta]$:
 \begin{align}\label{wt-bound-loc}
   & 0\leq W(t)=\|\rho(t,\cdot)\|_{L^1(0,1)}\leq M,
     \\\label{rho-bound-loc}
   &  \|\rho(t,\cdot)\|_{L^{\infty} (0,1)}
     \leq  e^{T\|\lambda_x\|_{C^0}}
        \cdot \max \Big\{ \|\rho_0\|_{L^{\infty}},
                 \f{\|u\|_{L^{\infty}}}{\ov \lambda(M)}\Big\}.
 \end{align}
\end{lem}

\begin{lem} \label{uni-sol} The weak solution to Cauchy problem
\eqref{eq}, \eqref{eq-IC} and \eqref{eq-BC} is unique.
\end{lem}

We leave the proofs of Lemma  \ref{rho-sol} and Lemma \ref{uni-sol}
in Appendix.

Now we suppose that we have solved Cauchy problem \eqref{eq},
\eqref{eq-IC} and \eqref{eq-BC} to the moment $\tau \in (0,T)$ with
the weak solution $\rho\in C^0([0,\tau];L^p(0,1)) \cap
L^{\infty}((0,\tau)\times (0,1))$. Similar to Lemma \ref{rho-sol}
and Lemma \ref{uni-sol}, we know that this weak solution is given by
 \beq
 \rho(t,x)=
   \begin{cases}
     \rho_0(\beta)\,  e^{-\int_0^t \lambda_x(\xi_4(\theta),W(\theta)) d\theta },\quad
   & \text{if} \  0\leq \xi_2(t)\leq x \leq 1,0\leq t\leq \tau,
       \\
    \displaystyle \f{u(\alpha)}{\lambda(0,W(\alpha))}
     \, e^{-\int_{\alpha}^t \lambda_x(\xi_3(\theta),W(\theta)) d\theta}, \quad
   &   \text{else}.
  \end{cases}
 \eeq
Moreover, the two uniform a priori estimates \eqref{wt-bound-loc}
and \eqref{rho-bound-loc} hold for all $t\in [0,\tau]$. Hence we can
choose $\delta\in (0,T)$ independent of $\tau$ such that
\eqref{delta} holds. Applying Lemma \ref{rho-sol} and Lemma
\ref{uni-sol} again, the weak solution $\rho\in C^0([0,\tau];
L^p(0,1))$, as well as estimates \eqref{wt-bound-loc} and
\eqref{rho-bound-loc}, is extended to the time interval
$[\tau,\tau+\delta] \cap [\tau,T]$. Step by step, we finally have a
unique global weak solution $\rho\in C^0([0,T];L^p(0,1)) \cap
L^{\infty}((0,T)\times (0,1))$. This finishes the proof of Theorem
\ref{thm-sol}.
\end{proof}

\begin{rem}\label{rem-rho}
Let $\rho$ be the weak solution in Theorem \ref{thm-sol} and $W\in
C^0([0,T])$ be the total mass function: $W(t)=\int_0^1\rho(t,x)dx$.
Let $\xi_1, \xi_2,\xi_3,\xi_4$ and $\alpha, \beta$ be defined by
\eqref{xi1}, \eqref{xi2}, \eqref{xi3}, \eqref{xi4} and
\eqref{alp-bet}, respectively. It follows from our proof of Theorem
\ref{thm-sol} that (see Fig \ref{Fig3}, Fig \ref{Fig4} and Fig
\ref{Fig5})
 \be \label{rho}
 \rho(t,x)=
   \begin{cases}
     \rho_0(\beta)\,  e^{-\int_0^t \lambda_x(\xi_4(\theta),W(\theta)) d\theta },\quad
   & \text{if} \  0\leq \xi_2(t)\leq x \leq 1,0\leq t\leq T,
       \\
    \displaystyle \f{u(\alpha)}{\lambda(0,W(\alpha))}
     \, e^{-\int_{\alpha}^t \lambda_x(\xi_3(\theta),W(\theta)) d\theta}, \quad
   & \text{else}.
  \end{cases}
 \ee
and the following estimate holds:
  \be  \label{rho-bound} \|\rho(t,\cdot)\|_{L^{\infty} (0,1)}
   \leq  e^{T\|\lambda_x\|_{C^0}}
   \cdot \max \Big\{ \|\rho_0\|_{L^{\infty}},
                 \f{\|u\|_{L^{\infty}}}{\ov \lambda(M)}\Big\},
   \quad \forall t\in [0,T]. \ee
Moreover, $W(t)$ can be expressed as (see Fig \ref{Fig5})
 \be \label{wt}
     W(t)=
 \begin{cases}\displaystyle
    \int_0^t u(\alpha) d\alpha
     +\int_0^{1-\int_0^t \lambda(\xi_1(\theta),W(\theta))d\theta}
        \rho_0(\beta)d\beta,\quad & 0\leq t\leq
       \xi_2^{-1}(1),
   \\\displaystyle
      \int_{\xi_1^{-1}(0)}^t u(\alpha) d\alpha,
    &  \xi_2^{-1}(1) \leq t\leq  T,
 \end{cases}
 \ee
which implies again that
 \be \label{wt-bound}
 0 \leq W(t)=\|\rho(t,\cdot)\|_{L^1(0,1)} \leq  M,\quad \forall t\in [0,T].
 \ee

Finally, $W$ is Lipschitz continuous:
 \beq  W(t)=W(0)+\int_0^t W'(s)ds,  \eeq
where (see Fig \ref{Fig5})
 \beq
     W'(t)=
 \begin{cases}\displaystyle
     u(t)-\lambda(\xi_1(t),W(t)) \rho_0(1-\int_0^t \lambda(\xi_1(\theta),W(\theta))d\theta),\quad
     & 0\leq t\leq  \xi_2^{-1}(1),
  \\ \displaystyle
     u(t)-u(\xi_1^{-1}(0))\f{\lambda(1,W(t))}{\lambda(0,W(\xi_1^{-1}(0)))}
     \,e^{-\int_{\xi_1^{-1}(0)}^t\lambda_x(\xi_1(\theta),W(\theta))\, d\theta},
    &  \xi_2^{-1}(1)  \leq t\leq  T
\end{cases}
 \eeq
and
 \be \label{w't-bound}
     \|W'\|_{L^{\infty}(0,T)}
   \leq  \|u\|_{L^{\infty}}
         + \|\lambda\|_{C^0}
          \cdot \max \Big\{ \|\rho_0\|_{L^{\infty}},
             \f{\|u\|_{L^{\infty}}}{\ov \lambda(M)}\, e^{T\|\lambda_x\|_{C^0}} \Big\}
    < \infty.  \ee

\end{rem}

\begin{rem} \label{rem-hid-reg}
{\bf (Hidden regularity.)} From the definition of the weak solution,
we can expect $\rho\in L^{\infty}((0,T)\times (0,1))=
L^{\infty}(0,1; L^{\infty}(0,T))$. In fact, under the assumptions of
Theorem \ref{thm-sol}, we have the hidden regularity that $\rho\in
C^0([0,1];L^p(0,T))$ for all $p\in [1,\infty)$ so that the function
$t \mapsto \rho(t,x)\in L^p(0,T)$ is well defined for every fixed
$x\in [0,1]$. The proof of the hidden regularity is quite similar to
our proof of $\rho\in C^0([0,T];L^p(0,1))$ by means of the implicit
expressions \eqref{rho} for $\rho$ and \eqref{wt} for $W(t)$ (see
also \eqref{w't-bound} when $T$ is large).
\end{rem}

\begin{rem}
If $\rho_0\in C^0([0,1])$ and $u\in C^0([0,T])$ are nonnegative and
the $C^0$ compatibility condition is satisfied at the origin:
 \beq \f{u(0)}{\lambda(0,W(0))}-\rho_0(0)=0, \eeq
where $W(0)=\int_0^1\rho_0(x)dx,$
then Cauchy problem \eqref{eq}, \eqref{eq-IC} and \eqref{eq-BC}
admits a unique nonnegative solution $\rho\in C^0([0,T]\times
[0,1])$. If, furthermore, $\rho_0\in C^1([0,1])$ and $u\in
C^1([0,T])$ are nonnegative and the $C^1$ compatibility conditions
are satisfied at the origin:
 \beq \begin{cases}
        \displaystyle
    \f{u(0)}{\lambda(0,W(0))}-\rho_0(0)=0,\\
       \displaystyle
    \f{u'(0)\lambda(0,W(0))-u(0)\lambda_W(0,W(0))W'(0)}{|\lambda(0,W(0))|^2}
     +\lambda(0,W(0))\rho_0'(0)+\lambda_x(0,W(0))\rho_0(0)=0,
  \end{cases}
 \eeq
where  $ W(0)=\int_0^1\rho_0(x)dx$ and $W'(0)=u(0)-
\rho_0(1)\lambda(1,W(0))$, then Cauchy problem \eqref{eq},
\eqref{eq-IC} and \eqref{eq-BC} admits a unique nonnegative
classical solution $\rho\in C^1([0,T]\times [0,1])$.
\end{rem}

\section{Stability with respect to the initial and boundary data}

In this section, we study the stability (or continuous dependence)
of both the solution $\rho$ itself and the out-flux $y$ with respect
to $\rho_0$ and $u$. That is to say: if the initial and boundary
data are slightly perturbed, are the solution $\rho$ and the
out-flux $y$ \emph{also slightly perturbed}?

Let $\ov \rho$ be the weak solution to the Cauchy problem with the
perturbed initial and boundary conditions
 \be \label{ov rho-eqn}
 \begin{cases}
 \ov\rho_t(t,x)+(\ov\rho(t,x)\lambda(x,\ov W(t)))_x=0, \quad &
   t\geq 0, 0\leq x\leq 1, \\
 \ov\rho(0,x)=\ov\rho_0(x),\quad &0\leq x\leq 1,\\
  \ov\rho(t,0)\lambda(0,\ov W(t))=\ov u(t), \quad &0\leq t\leq T,
 \end{cases}
 \ee
where $\ov W(t):=\int_0^1 \ov \rho(t,x)dx$.  We denote that $\ov
y(t):=\ov\rho(1,t)\lambda(1,\ov W(t))$. We also define the
corresponding characteristics with respect to the perturbed solution
$\ov \rho$\,: $\ov\xi_1$ (as \eqref{xi1}), $\ov\xi_2$ (as
\eqref{xi2}), $(\ov\xi_3,\ov\alpha)$ (as \eqref{xi3} and
\eqref{alp-bet}) and $(\ov\xi_4,\ov\beta)$ (as \eqref{xi4} and
\eqref{alp-bet}), respectively.

First we have the following theorem on the stability of the weak
solution $\rho$.

\begin{thm} \label{thm-stab-sol}
For any $\vep>0$, $p\in [1,\infty)$ and any $K>0$ such that
 \be  \label{K}
   \|\rho_0\|_{L^{\infty}(0,1)} + \|u\|_{L^{\infty}(0,T)} \leq K,
     \quad
  \|\ov \rho_0\|_{L^{\infty}(0,1)} + \|\ov u\|_{L^{\infty}(0,T)} \leq K, \ee
there exists $\eta=\eta(\vep,p,K)>0$ small enough such that, if
 \be\label{IBC-small}
   \|\ov \rho_0 -\rho_0\|_{L^p(0,1)}
  +\|\ov{u}-u\|_{L^p(0,T)}< \eta,
  \ee
then
 \be \label{stab-sol}\|\ov \rho(t,\cdot) -\rho(t,\cdot)\|_{L^p(0,1)}< \vep,
  \quad \forall t\in [0,T]. \ee
\end{thm}

\begin{proof}
Solving Cauchy problem \eqref{ov rho-eqn}, we know from
\eqref{wt-bound} and \eqref{K} that
  \be \label{ov w-bound}
   0 \leq \ov W(t) \leq  \|\ov\rho_0\|_{L^1(0,1)}+\|\ov u\|_{L^1(0,T)} \leq K,
   \quad \forall t\in [0,T].
 \ee
Replacing $M$ by $K$ in the definitions of $\ov\lambda(M)$ and
$\|\lambda\|_{C^0}, \|\lambda_x\|_{C^0}, \|\lambda_W\|_{C^0}$ (see
Section 2), we introduce some new notations as $\ov\lambda(K)$ and
(still) $\|\lambda\|_{C^0}, \|\lambda_x\|_{C^0},
\|\lambda_W\|_{C^0}$ in this section.

We first prove the stability of the weak solution for small time.
Let $\delta$ be chosen by \eqref{delta}. For any fixed
$t\in[0,\delta]$, we suppose that $\xi_2(t)<\ov\xi_2(t)$ (the case
that $\xi_2(t) \geq \ov\xi_2(t)$ can be treated similarly). In order
to estimate $\|\ov \rho(t,\cdot) -\rho(t,\cdot)\|_{L^p(0,1)}$ for
$p\in [1,\infty)$, we need to estimate
$\int_0^{\xi_2(t)}|\ov\rho(t,x)-\rho(t,x)|^p dx$,
$\int_{\xi_2(t)}^{\ov \xi_2(t)}|\ov\rho(t,x)-\rho(t,x)|^p dx$ and
$\int_{\ov\xi_2(t)}^1|\ov\rho(t,x)-\rho(t,x)|^p dx$, successively.

For almost every $x\in [0,\xi_2(t)]$, we know from Remark
\ref{rem-rho} that (see Fig \ref{Fig6})
  \begin{align*}
   &\quad |\ov\rho(t,x)-\rho(t,x)|
         \\
   &=\Big|\f{\ov u(\ov\alpha)}{\lambda(0,\ov W(\ov\alpha))}
          \, e^{-\int_{\ov\alpha}^t\lambda_x(\ov\xi_3(\theta),\ov W(\theta))\,d\theta}
      -\f{u(\alpha)}{\lambda(0,W(\alpha))}
          \, e^{-\int_{\alpha}^t\lambda_x(\xi_3(\theta),W(\theta))\,d\theta}\Big|
         \\
   & \leq \f{|\ov u(\ov\alpha)- u(\alpha)| }{\lambda(0,\ov W(\ov\alpha))}
           \, e^{-\int_{\ov\alpha}^t\lambda_x(\ov\xi_3(\theta),\ov W(\theta))\,d\theta}
        +  \Big| \f{u( \alpha)}{\lambda(0,\ov W(\ov\alpha))}
                             -\f{u( \alpha)}{\lambda(0, W(\alpha))}\Big|
            \, e^{-\int_{\ov\alpha}^t\lambda_x(\ov\xi_3(\theta),\ov W(\theta))\,d\theta}
         \\
   & \quad + \f{|u(\alpha)|}{\lambda(0,W(\alpha))}
            \Big| e^{-\int_{\ov\alpha}^t\lambda_x(\ov\xi_3(\theta),\ov W(\theta))\,d\theta}
                  -e^{-\int_{\alpha}^t\lambda_x(\xi_3(\theta),W(\theta))\,d\theta} \Big|
         \\
   &\leq C |\ov u(\ov\alpha)-u(\alpha)| + C |u(\alpha)| |\ov W(\ov\alpha)- W(\alpha)|
         \\
   &\quad  +C |u(\alpha)|  \Big|\int_{\ov\alpha}^t\lambda_x(\ov\xi_3(\theta),\ov W(\theta))\,d\theta
      -\int_{\alpha}^t\lambda_x(\xi_3(\theta),W(\theta))\,d\theta\Big|.
   \end{align*}
Here  and hereafter in this section, we denote by $C$ various
constants which do not depend on $t,x,\rho,\ov\rho$ but may depend
on $p$ and $K$.

For the given $u\in L^{\infty}(0,T)$, let $\{u_n\}_{n=1}^{\infty}
\subset C^1([0,T])$ be such that $u_n \rightarrow u$ in $L^p(0,T)$.
And for the given $\lambda\in C^1([0,1]\times [0,\infty))$, let
$\{v_n\}_{n=1}^{\infty} \subset C^1([0,1] \times [0,K])$ be such
that $v_n \rightarrow \lambda_x$ in $C^0([0,1] \times [0,K])$. Using
sequences $\{u_n\}_{n=1}^{\infty}$, $\{v_n\}_{n=1}^{\infty}$ and
\eqref{w't-bound-loc}, we obtain for almost every $x\in
[0,\xi_2(t)]$ that
  \begin{align}\label{stab-1-1}
   &\quad   |\ov\rho(t,x)-\rho(t,x)|
      \nonumber  \\
   &\leq C |\ov u(\ov\alpha)-u(\ov\alpha)|+C |u_n(\ov\alpha)-u(\ov\alpha)|
        +C |u_n(\alpha)-u(\alpha)|
      \nonumber \\
   &\quad + C |u_n(\ov\alpha)-u_n(\alpha)|
          + C |\ov W(\ov \alpha)-W(\ov \alpha)|
           +C |W(\ov \alpha)-W(\alpha)|
      \nonumber \\
   &\quad  +C |u(\alpha)| \int_{\ov\alpha}^t |v_n(\ov\xi_3(\theta),\ov W(\theta))
            -\lambda_x(\ov\xi_3(\theta),\ov W(\theta))| \,d\theta
      \nonumber \\
   &\quad  +C |u(\alpha)|  \int_{\alpha}^t |v_n(\xi_3(\theta),W(\theta))
           -\lambda_x(\xi_3(\theta),W(\theta))| \,d\theta
      \nonumber \\
   &\quad +C |u(\alpha)|  \Big|\int_{\ov\alpha}^t v_n(\ov\xi_3(\theta),\ov W(\theta))\,d\theta
      -\int_{\alpha}^t v_n(\xi_3(\theta),W(\theta))\,d\theta\Big|
                 \nonumber  \\
   &\leq C |\ov u(\ov\alpha)-u(\ov\alpha)|+C |u_n(\ov\alpha)-u(\ov\alpha)|
        +C |u_n(\alpha)-u(\alpha)|
      \nonumber \\
    &\quad  +C |u(\alpha)|\|v_n -\lambda_x\|_{C^0([0,1]\times [0,K])}
            + C_n |\ov\alpha-\alpha|
            + C_n |u(\alpha)| |\ov\alpha-\alpha|
      \nonumber \\
   &\quad   + C_n |u(\alpha)| \|\ov\xi_3-\xi_3\|_{C^0([\ov\alpha,t])}
      + C_n |u(\alpha)|  \|\ov W -W\|_{C^0([0,\delta])}.
\end{align}
Here  and hereafter in this section, we denote by $C_n$ various
constants which do not depend on $t,x,\rho,\ov\rho$ but may depend
on $p,K$ and $n$ (the index of the corresponding sequences, e.g.
$\{u_n\}_{n=1}^{\infty}$, $\{v_n\}_{n=1}^{\infty}$ and so on).

A similar estimate as \eqref{ov xi1-xi1} gives us
 \be\label{ov xi3-xi3}
   \|\ov\xi_3-\xi_3\|_{C^0([\ov\alpha,t])}
  \leq  \f{t \|\lambda_W\|_{C^0}}{1-t\|\lambda_x\|_{C^0}}
  \cdot \|\ov W -W\|_{C^0([0,\delta])}. \ee
From the fact that
   \beq \ov\xi_3(t)=\int_{\ov\alpha}^t\lambda(\ov\xi_3(\theta),\ov W(\theta))\,d\theta
    =x=\xi_3(t)=\int_{\alpha}^t\lambda (\xi_3(\theta),W(\theta))\,d\theta
  \eeq
and  the definition of $\ov \lambda(K)$, we get also that
  \begin{align} \label{ov alpha-alpha}
    |\ov\alpha-\alpha|
  &\leq \f{1}{\ov\lambda(K)} \int_{\alpha}^{\ov\alpha} \lambda(\xi_3(\theta),W(\theta))\,d\theta
     \nonumber\\
  &=\f{1}{\ov\lambda(K)} \int_{\ov\alpha}^t ( \lambda(\ov\xi_3(\theta),\ov W(\theta))
     -\lambda(\xi_3(\theta),W(\theta)) )\,d\theta
     \nonumber\\
   &\leq \f{1}{\ov\lambda(K)}
     \Big( t\|\lambda_x\|_{C^0} \|\ov\xi_3-\xi_3\|_{C^0([\ov\alpha,t])}
      + t \|\lambda_W\|_{C^0} \|\ov W -W\|_{C^0([0,\delta])}  \Big)
       \nonumber\\
   &\leq \f{t \|\lambda_W\|_{C^0}}{\ov\lambda(K)(1-t\|\lambda_x\|_{C^0})}
          \cdot \|\ov W -W\|_{C^0([0,\delta])}.
  \end{align}
On the other hand, by \eqref{wt} and H\"{o}lder inequality, we have
for every $t\in [0,\delta]$ that
 \begin{align*}
 &\quad |\ov W(t)-W(t)|
   \\
 &=\Big |\int_0^t (\ov u(\sigma)-u(\sigma)) d\sigma
   +\int_0^{1-\int_0^t\lambda(\ov\xi_1(\theta),\ov W(\theta))\,d\theta}\ov
   \rho_0(\sigma)\,d\sigma
       -\int_0^{1-\int_0^t\lambda(\xi_1(\theta),W(\theta))\,d\theta}\rho_0(\sigma)\,d\sigma\Big |
    \\
 &\leq \delta^{\f 1q} \|\ov u-u\|_{L^p(0,T)} +\|\ov\rho_0-\rho_0\|_{L^p(0,1)}
   +\|\rho_0\|_{L^{\infty}}
     \int_0^t |\lambda(\ov\xi_1(\theta),\ov W(\theta))-\lambda(\xi_1(\theta),W(\theta))|\,d\theta
   \\
 &\leq \delta^{\f 1q} \|\ov u-u\|_{L^p(0,T)}+\|\ov\rho_0-\rho_0\|_{L^p(0,1)}
   \\
 &\quad   +\delta  \|\rho_0\|_{L^{\infty}}(\|\lambda_x\|_{C^0} \|\ov\xi_1-\xi_1\|_{C^0([0,t])}
   + \|\lambda_W\|_{C^0} \|\ov W -W\|_{C^0([0,\delta])} ),
 \end{align*}
where $q$ satisfies $\f 1p+ \f 1q=1$.
By the definitions of $\xi_1$ and $\ov\xi_1$, we still have
\eqref{ov xi1-xi1}, thus
  \begin{align*}  &\quad  \|\ov W -W\|_{C^0([0,\delta])}
  = \sup_{t\in [0,\delta]}|\ov W(t)-W(t)|
      \\
  & \leq
  \f{1-\delta \|\lambda_x\|_{C^0}}
  {1-\delta (\|\lambda_x\|_{C^0}+\|\rho_0\|_{L^{\infty}} \|\lambda_W\|_{C^0})}
   \cdot (\delta^{\f 1q} \|\ov  u-u\|_{L^p(0,T)} +\|\ov\rho_0-\rho_0\|_{L^p(0,1)}).
   \end{align*}
and furthermore, from the choice \eqref{delta} of $\delta$,
  \be \label{ov w-w}  \|\ov W -W\|_{C^0([0,\delta])}\leq
  C \|\ov u-u\|_{L^p(0,T)} +C \|\ov\rho_0-\rho_0\|_{L^p(0,1)}.
   \ee
Therefore, combining \eqref{stab-1-1}, \eqref{ov xi3-xi3}, \eqref{ov
alpha-alpha}  and \eqref{ov w-w} all together and using Lemma
\ref{jacobi1}, we obtain easily that
  \begin{align}\label{stab-1-2}
  &\quad  \int_0^{\xi_2(t)}|\ov\rho(t,x)-\rho(t,x)|^p dx
         \nonumber \\
   & \leq C \|u_n-u\|^p_{L^p(0,T)}
          + C\|v_n-\lambda_x\|^p_{C^0([0,1] \times [0,K])}
      \nonumber \\
    &\quad  +C_n \|\ov u-u\|^p_{L^p(0,T)}
              + C_n \|\ov\rho_0-\rho_0\|^p_{L^p(0,1)}.
            \end{align}

For almost every $x\in [\ov\xi_2(t),1]$,   we know from Remark
\ref{rem-rho} that (see Fig \ref{Fig7})
  \begin{align*}
     &\quad |\ov\rho(t,x)-\rho(t,x)|=\Big
          |\ov\rho_0(\ov{\beta})e^{-\int_0^t\lambda_x(\ov\xi_4(\theta),\ov W(\theta))\, d\theta}
            -\rho_0(\beta)  e^{-\int_0^t\lambda_x(\xi_4(\theta),W(\theta))\,  d\theta}\Big |
     \\
    & \leq |\ov\rho_0(\ov{\beta})-\rho_0(\beta)|
       e^{-\int_0^t\lambda_x(\ov\xi_4(\theta),\ov W(\theta))\, d\theta}
     +|\rho_0(\beta)|  \Big| e^{-\int_0^t\lambda_x(\ov\xi_4(\theta),\ov W(\theta))\, d\theta}
        -e^{-\int_0^t\lambda_x(\xi_4(\theta),W(\theta))\, d\theta}\Big |
        \\
  &\leq C |\ov\rho_0(\ov \beta)-\rho_0(\ov\beta)|
     +C|\rho_0(\ov\beta)-\rho_0(\beta)|
       \\
  &\quad  +C |\rho_0(\beta)| \Big| \int_0^t (\lambda_x(\ov\xi_4(\theta),\ov W(\theta))
             -\lambda_x(\xi_4(\theta),W(\theta)))\, d\theta\Big|.
  \end{align*}

For the given $\rho_0\in L^{\infty}(0,1)$, we let
$\{\rho_0^n\}_{n=1}^{\infty} \subset C^1([0,1])$ be such that
$\rho_0^n \rightarrow \rho_0$ in $L^p(0,1)$. With the help of
sequences $\{\rho^n_0\}_{n=1}^{\infty}$, $\{v_n\}_{n=1}^{\infty}$,
we obtain for almost every $x\in [\ov\xi_2(t),1]$ that
 \begin{align} \label{stab-2-1}
 & \quad  |\ov\rho(t,x)-\rho(t,x)|
   \nonumber\\
 &\leq C |\ov\rho_0(\ov \beta)-\rho_0(\ov\beta)|
     +C|\rho_0^n(\ov\beta)-\rho_0(\ov\beta)|
     +C|\rho_0^n(\beta)-\rho_0(\beta)|
     + C|\rho_0^n(\ov\beta)-\rho_0^n(\beta)|
   \nonumber\\
 &\quad  +C |\rho_0(\beta)|  \int_0^t |v_n(\ov\xi_4(\theta),\ov W(\theta))
              -\lambda_x(\ov\xi_4(\theta),\ov W(\theta))|\, d\theta
   \nonumber\\
 & \quad +C |\rho_0(\beta)| \int_0^t |v_n(\xi_4(\theta), W(\theta))
              -\lambda_x(\xi_4(\theta),W(\theta))| \, d\theta
    \nonumber\\
  &\quad +C |\rho_0(\beta)| \int_0^t |v_n(\ov \xi_4(\theta), \ov W(\theta))
              -v_n(\xi_4(\theta),W(\theta))| \, d\theta
     \nonumber\\
 &\leq C |\ov\rho_0(\ov \beta)-\rho_0(\ov\beta)|
     +C|\rho_0^n(\ov\beta)-\rho_0(\ov\beta)|
     +C|\rho_0^n(\beta)-\rho_0(\beta)|
     + C |\rho_0(\beta)| \|v_n-\lambda_x\|_{C^0([0,1]\times[0,K])}
     \nonumber\\
  &\quad   + C_n|\ov\beta -\beta|
     +C_n |\rho_0(\beta)| \|\ov\xi_4-\xi_4\|_{C^0([0,t])}
    +C_n |\rho_0(\beta)| \|\ov W -W\|_{C^0([0,\delta])}.
\end{align}
Similar to \eqref{ov xi1-xi1}, we have
 \be \label{ov xi4-xi4}
 \|\ov\xi_4-\xi_4\|_{C^0([0,t])}
 \leq\f{t\|\lambda_W\|_{C^0}}{1-t \|\lambda_x\|_{C^0}}
 \cdot \|\ov W -W\|_{C^0([0,\delta])},
 \ee
 and in particular,
 \be \label{ov beta-beta}
 |\ov\beta-\beta|= |\ov\xi_4(0)-\xi_4(0)|
 \leq\f{t\|\lambda_W\|_{C^0}}{1-t \|\lambda_x\|_{C^0}}
 \cdot \|\ov W -W\|_{C^0([0,\delta])}.
 \ee
Therefore, by the choice \eqref{delta} of $\delta$, together with
estimates  \eqref{ov w-w}, \eqref{stab-2-1},\eqref{ov xi4-xi4},
\eqref{ov beta-beta} and Lemma \ref{jacobi2}, we get immediately
that
\begin{align}\label{stab-2-2}
 &\quad   \int_{\ov \xi_2(t)}^1 |\ov\rho(t,x)-\rho(t,x)|^p dx
    \nonumber\\
 & \leq C\|\rho_0^n-\rho_0\|^p_{L^p(0,1)}
       +C\|v_n-\lambda_x\|^p_{C^0([0,1]\times [0,K])}
   \nonumber\\
  &\quad      +C_n\|\ov u-u\|^p_{L^p(0,T)}
       +C_n \|\ov\rho_0-\rho_0\|^p_{L^p(0,1)}.
\end{align}

Now it is left to estimate $\int_{\xi_2(t)}^{\ov\xi_2(t)}
|\ov\rho(t,x)-\rho(t,x)|^p dx$ (see Fig \ref{Fig8}).  Obviously, we
have
 \be \label{stab-3-1}
  \int_{\xi_2(t)}^{\ov\xi_2(t)}|\ov\rho(t,x)-\rho(t,x)|^p dx \leq
  \int_{\xi_2(t)}^x |\ov\rho(t,x)-\rho(t,x)|^p dx
  +\int_x^{\ov\xi_2(t)}|\ov\rho(t,x)-\rho(t,x)|^p dx.
  \ee
Then similar to estimates \eqref{stab-1-2} and \eqref{stab-2-2}, we
get easily from \eqref{stab-3-1} that
 \begin{align} \label{stab-3-2}
  & \quad \int_{\xi_2(t)}^{\ov\xi_2(t)}|\ov\rho(t,x)-\rho(t,x)|^p dx
    \nonumber \\
    &\leq C\|\rho_0^n -\rho_0\|^p_{L^p(0,1)}
        + C \|u_n-u\|^p_{L^p(0,T)}
        +C\|v_n-\lambda_x\|^p_{C^0([0,1]\times [0,K])}
     \nonumber\\
    & \quad
         +C_n\|\ov u-u\|^p_{L^p(0,T)}
         +C_n \|\ov\rho_0 -\rho_0\|^p_{L^p(0,1)}.
  \end{align}

Summarizing  estimates \eqref{stab-1-2},\eqref{stab-2-2} and
\eqref{stab-3-2}, finally we prove for any fixed $t\in[0,\delta]$
with $\delta$ satisfying \eqref{delta} that
   \begin{align} \label{stab-sol-2}
    & \quad \|\ov\rho(t,\cdot)-\rho(t,\cdot)\|^p_{L^p(0,1)}
       \nonumber\\
    &\leq C\|\rho_0^n -\rho_0\|^p_{L^p(0,1)}
        + C \|u_n-u\|^p_{L^p(0,T)}
        +C\|v_n-\lambda_x\|^p_{C^0([0,1]\times [0,K])}
        \nonumber\\
    & \quad
         +C_n\|\ov u-u\|^p_{L^p(0,T)}
         +C_n \|\ov\rho_0 -\rho_0\|^p_{L^p(0,1)}.
\end{align}
  Now if we take $n$ large enough and then $\eta$ in \eqref{IBC-small}
small enough, the right-hand side of \eqref{stab-sol-2} can be
smaller than any given constant $\vep>0$.

In order to obtain the global stability \eqref{stab-sol} from
\eqref{stab-sol-2}, it suffices, according to \eqref{delta}, to
prove uniform a priori estimates for
$\|\rho(t,\cdot)\|_{L^1{(0,1)}}$,
$\|\rho(t,\cdot)\|_{L^{\infty}(0,1)}$,
$\|\ov\rho(t,\cdot)\|_{L^1{(0,1)}}$, and
$\|\ov\rho(t,\cdot)\|_{L^{\infty}(0,1)}$. In fact the desired a
priori estimates are available by Remark \ref{rem-rho} (see
estimates \eqref{rho-bound} and \eqref{wt-bound}), hence we finally
reach \eqref{stab-sol}. This finishes the proof of Theorem
\ref{thm-stab-sol}.

\end{proof}

We also have the stability on the out-flux $y$.

\begin{thm} \label{thm-stab-flux}
For any $\vep>0$, $p\in [1,\infty)$ and any $K$ such that \eqref{K}
holds, there exists $\eta=\eta(\vep,p,K))>0$ small enough such that,
if \eqref{IBC-small} holds, then
 \be \label{stab-flux}\|\ov y -y\|_{L^p(0,T)}< \vep. \ee
\end{thm}

\begin{proof} First, we prove $\|\ov y -y\|_{L^p(0,\delta)}<
\vep$ for $\delta$ small. Let $\delta$ be such that \eqref{delta}
holds. For any fixed $t\in [0,\delta]$, without loss of generality,
we suppose that $\xi_2(t)<\ov\xi_2(t)$ (the case that $\xi_2(t)\geq
\ov\xi_2(t)$ can be treated similarly).
By Remark \ref{rem-rho}, we have for almost every $t\in [0,\delta]$
that (see Fig \ref{Fig9})
  \begin{align*}
  &\quad |\ov y(t)-y(t)|
      \nonumber\\
  &=|\ov\rho(t,1)\lambda(1,\ov W(t))-\rho(t,1)\lambda(1,W(t))|
      \nonumber\\
  &\leq C |\ov\rho(t,1)-\rho(t,1)| + C|\rho(t,1)|\|\ov W -W\|_{C^0([0,\delta])}
       \nonumber\\
  &=C \Big|\ov \rho_0(\ov \xi_1(0))
             e^{-\int_0^t \lambda_x(\ov \xi_1(\theta),\ov W(\theta)) d\theta}
       - \rho_0(\xi_1(0))
             e^{-\int_0^t \lambda_x(\xi_1(\theta),W(\theta)) d\theta} \Big|
      \nonumber\\
  &\quad   +C |\rho(t,1)| \|\ov W -W\|_{C^0([0,\delta])}
  \nonumber\\
  &\leq C |\ov \rho_0(\ov \xi_1(0))-\rho_0(\xi_1(0))|
        + C |\rho(t,1)| \|\ov W -W\|_{C^0([0,\delta])}
     \nonumber\\
  &\quad + C |\rho_0(\xi_1(0))|\Big| \int_0^t ( \lambda_x(\ov \xi_1(\theta),\ov W(\theta))
              -\lambda_x(\xi_1(\theta),W(\theta)) )  d\theta \Big|.
  \end{align*}
Then using  sequences $\{\rho_0^n\}_{n=1}^{\infty}$,
$\{v_n\}_{n=1}^{\infty}$, we obtain
  \begin{align*}
  &\quad |\ov y(t)-y(t)|
             \nonumber\\
  &\leq C |\ov \rho_0(\ov \xi_1(0))-\rho_0(\ov \xi_1(0))|
        + C |\rho_0^n(\ov \xi_1(0))-\rho_0(\ov \xi_1(0))|
        + C |\rho_0^n(\xi_1(0))-\rho_0(\xi_1(0))|
              \nonumber\\
  & \quad   +C |\rho_0^n(\ov \xi_1(0))-\rho_0^n(\xi_1(0))|
           + C |\rho(t,1)| \|\ov W -W\|_{C^0([0,\delta])}
              \nonumber\\
  & \quad    + C |\rho_0(\xi_1(0))|  \int_0^t |v_n(\ov \xi_1(\theta),\ov W(\theta))
              -v_n( \xi_1(\theta), W(\theta)) | d\theta
       \nonumber\\
   &\quad  + C |\rho_0(\xi_1(0))|  \int_0^t |v_n(\ov \xi_1(\theta),\ov W(\theta))
              -\lambda_x(\ov \xi_1(\theta),\ov W(\theta)) | d\theta
       \nonumber\\
   &\quad  +C |\rho_0(\xi_1(0))| \int_0^t |v_n(\xi_1(\theta),W(\theta))
              -\lambda_x(\xi_1(\theta),W(\theta)) | d\theta,
       \nonumber\\
   &\leq C |\ov \rho_0(\ov \xi_1(0))-\rho_0(\ov \xi_1(0))|
        + C |\rho_0^n(\ov \xi_1(0))-\rho_0(\ov \xi_1(0))|
      \nonumber\\
    &\quad     + C |\rho_0^n(\xi_1(0))-\rho_0(\xi_1(0))|
           +C |\rho_0(\xi_1(0))| \|v_n-\lambda_x\|_{C^0([0,1]\times [0,K])}
              \nonumber\\
   &  \quad  + C_n  |\rho_0(\xi_1(0))| \|\ov \xi_1-\xi_1\|_{C^0([0,t])}
            + C_n |\rho_0(\xi_1(0))| \|\ov W -W\|_{C^0([0,\delta])},
\end{align*}
which results in
  \begin{align*}
     &\quad \int_0^{\delta}|\ov y(t)-y(t)|^p \,dt
          \\
     &\leq C\|\rho_0^n -\rho_0\|^p_{L^p(0,1)}
         +C\|v_n-\lambda_x\|^p_{C^0([0,1]\times [0,K])}
         + C_n  \|\ov \rho_0-\rho_0\|^p_{L^p(0,1)}
        + C_n \|\ov u-u\|^p_{L^p(0,T)}
        \end{align*}
with the help of \eqref{ov xi1-xi1} and \eqref{ov w-w}.
Obviously, $\|\ov y-y\|_{L^p(0,\delta)}<\vep$ holds if we let $n$
large enough and then $\eta>0$ in \eqref{IBC-small} small enough.

In a same way, we can prove for every $\tau\in [0,T-\delta]$ by
applying Theorem \ref{thm-stab-sol} that
  \beq    \|\ov y-y\|_{L^p(\tau,\tau+\delta)}<\vep
  \eeq
if $\delta$ satisfies \eqref{delta} (independent of $\tau$). Step by
step, we finally prove \eqref{stab-flux}.

\end{proof}

\section{$L^p$-optimal control for demand tracking problem}

Let $T>0$ and $\rho_0\in L^{\infty}_+(0,1)$ be given. According to
Theorem \ref{thm-sol}, for every  $u\in L^{\infty}_{+}(0,T)$, Cauchy
problem \eqref{eq},\eqref{eq-IC} and \eqref{eq-BC} admits a unique
weak solution $L^{\infty}((0,T)\times (0,1)) \cap \rho\in
C^0([0,T],L^p(0,1)) \cap C^0([0,1],L^p(0,T))$ for all $p\in
[1,\infty)$.

For any fixed given demand $y_d\in L^{\infty}(0,T)$, initial data
$\rho_0\in L^{\infty}_+(0,1)$ and $p\in [1,\infty)$, we define a
cost functional on $L^{\infty}_{+}(0,T)$ by
 \beq
 J_p(u):=\|u\|^2_{L^{\infty}(0,T)} + \|y-y_d\|^2_{L^p(0,T)},
 \quad u\in L^{\infty}_{+}(0,T),
 \eeq
where $y(t):=\rho(t,1)\lambda(1,W(t))$ is the out-flux corresponding
to the in-flux $u\in L^{\infty}_{+}(0,T)$ and initial data $\rho_0$.
This cost functional is motivated by \cite{CKWang, LaMarca} and the
existence of the solution to this optimal control problem is
obtained by the following theorem.

\begin{thm}\label{thm-opt}
The infimum of the functional $J_p$ in $L^{\infty}_{+}(0,T)$ with
$p\in [1,\infty)$ is achieved, i.e., there exists $u^{\infty}\in
L^{\infty}_{+}(0,T)$ such that
 \beq
 J_p(u^{\infty})=\inf_{u\in L^{\infty}_{+}(0,T)}J_p(u).
 \eeq
\end{thm}

\begin{proof}

Let $\{u^n\}_{n=1}^{\infty}\subset L^{\infty}_{+}(0,T)$ be a
minimizing sequence of the functional $J_p$, i.e.
 \beq
 \lim_{n\rightarrow \infty}J_p(u^n)=\inf_{u\in L^{\infty}_{+}(0,T)}J_p(u).
 \eeq
Then we have
 \be\label{un-bound}
 \|u^n\|_{L^{\infty}(0,T)}+\|y^n\|_{L^p(0,T)}\leq C, \quad \forall n\in
 \mathbb{Z}^+.
 \ee
Here and hereafter in this section, we denote by $C$ various
constants which do not depend on $n$ (the index of the sequences,
e.g. $\{u^n\}_{n=1}^{\infty}$, $\{y^n\}_{n=1}^{\infty}$ and so on).
The boundedness of $u^n\subset L^{\infty}_+(0,T)$ implies that there
exists $u^{\infty}\in L^{\infty}_{+}(0,T)$ and a subsequence of
$\{u^{n_{k}}\}_{k=1}^{\infty}$ such that $
u^{n_{k}}\xrightharpoonup[]{*} u^{\infty}$ in $L^{\infty}_{+}(0,T)$.
For simplicity, we still denote the subsequence as
$\{u^n\}_{n=1}^{\infty}$. And we let
 \be \label{wt-M}
 \wt M:=\|\rho_0\|_{L^1(0,1)}
 +\max\{\sup_{n\in \mathbb{Z}^+}\|u^n\|_{L^1(0,T)},\|u^{\infty}\|_{L^1(0,T)}\}
 <\infty. \ee

 Let $\rho^n$ be the weak solution to the Cauchy problem
 \beq
 \begin{cases}
 \rho^n_t(t,x)+(\rho^n(t,x)\lambda(x,W^n(t)))_x=0, \quad & t\geq 0, 0\leq x\leq 1,\\
  \rho^n(0,x)=\rho_0(x),\quad &0\leq x\leq 1,\\
  \rho^n(t,0)\lambda(0,W^n(t))=u^n(t), \quad &0\leq t\leq T,
 \end{cases}
 \eeq
where $ W^n(t):=\int_0^1\rho^n(t,x)dx$. Then for any fixed $t\in
[0,T]$,  we define $\xi_1^n$ by
 \beq
 \f{d\xi_1^n}{ds} =\lambda(\xi_1^n(s),W^n(s)),\quad
      \xi_1^n(t)=1 \ \text{for any fixed }\ t\in [0,T],
\eeq
and define $\xi_2^n$ by
 \beq  \f{d\xi_2^n}{ds} =\lambda(\xi_2^n(s),W^n(s)),\quad
      \xi_2^n(0)=0. \eeq
Thanks to \eqref{un-bound}, we know from \eqref{wt} and
\eqref{w't-bound} that
 \be\label{wn-bound}
 \|W^n\|_{W^{1,\infty}(0,T)}\leq C, \quad \forall n\in \mathbb{Z}^+,
 \ee
Then it follows from Arzel\`{a}-Ascoli Theorem that there exists
$\overline W^{\infty}\in C^0([0,T])$ and a subsequence
$\{W^{n_{l}}\}_{l=1}^{\infty}$ such that $ W^{n_{l}} \rightarrow
\overline W^{\infty}$ in $ C^0([0,T])$. Now we choose the
corresponding subsequence $\{u^{n_{l}}\}_{l=1}^{\infty}$ and again,
denote it as $\{u^n\}_{n=1}^{\infty}$. Thus we have $
u^n\xrightharpoonup[]{*} u^{\infty}$ in $L^{\infty}_+(0,T)$ and $W^n
\rightarrow \overline W^{\infty}$ in $C^0([0,T])$.

In view of \eqref{wt}, \eqref{wt-M} and \eqref{wn-bound}, there
exists a small $\delta>0$ depending only on $\wt M$  and independent
of $n$, such that
  \be \label{wn} W^n(t)=\int_0^t u^n(\alpha)d\alpha
   +\int_0^{1-\int_0^t\lambda(\xi_1^n(\theta),W^n(\theta))d\theta} \rho_0(\beta)d\beta,
   \quad  \forall t \in [0,\delta].
 \ee

For any fixed $t\in [0,T]$,  we define $\ov\xi_1^{\infty}$ by
 \beq
  \f{d\ov\xi_1^{\infty}}{ds}
     =\lambda(\ov\xi_1^{\infty}(s),\ov W^{\infty}(s)),
      \quad \ov\xi_1^{\infty}(t)=1 \ \text{for any fixed }\ t\in [0,T],
 \eeq
and define $\ov\xi_2^{\infty}$ by
  \beq \f{d\ov\xi_2^{\infty}}{ds}
     =\lambda(\ov\xi_2^{\infty}(s),\ov W^{\infty}(s)),
      \quad \ov\xi_2^{\infty}(0)=0.
 \eeq
 Obviously, $W^n \rightarrow \overline W^{\infty}$ in $C^0([0,T])$
implies that for any $t\in [0,\delta]$, $\xi_1^n\rightarrow \ov
\xi_1^{\infty}$ in $C^1([0,t])$ and $\xi_2^n\rightarrow \ov
\xi_2^{\infty}$ in $C^1([0,T])$.
Thus by passing to the limit $n\rightarrow \infty$ in \eqref{wn}, we
obtain
  \be \label{ov-wInf}\ov W^{\infty}(t)=\int_0^t u^{\infty}(\alpha)d\alpha
   +\int_0^{1-\int_0^t\lambda(\ov \xi_1^{\infty}(\theta),\ov W^{\infty}(\theta))d\theta}
    \rho_0(\beta)d\beta,
   \quad  \forall t \in [0,\delta].
 \ee

Let $\rho^{\infty}$ be the weak solution to the Cauchy problem
 \beq
 \begin{cases}
 \rho^{\infty}_t(t,x)+(\rho^{\infty}(t,x)\lambda(x,W^{\infty}(t)))_x=0,
   \quad & t\geq 0, 0\leq x\leq 1,\\
 \rho^{\infty}(0,x)=\rho_0(x),\quad &0\leq x\leq 1,\\
  \rho^{\infty}(t,0)\lambda(0,W^{\infty}(t))=u^{\infty}(t), \quad &0\leq t\leq T,
 \end{cases}
 \eeq
where $W^{\infty}(t):=\int_0^1\rho^{\infty}(t,x)dx.$ For any fixed
$t\in [0,T]$, we define $\xi_1^{\infty}$ by
 \beq
  \f{d\xi_1^{\infty}}{ds}
     =\lambda(\xi_1^{\infty}(s),W^{\infty}(s)),
      \quad \xi_1^{\infty}(t)=1
 \eeq
 and define $\xi_2^{\infty}$ by
 \beq   \f{d \xi_2^{\infty}}{ds}
     =\lambda(\xi_2^{\infty}(s), W^{\infty}(s)),
      \quad \xi_2^{\infty}(0)=0.
      \eeq
Let  $\delta>0$ be small enough so that
  \be \label{wInf} W^{\infty}(t)=\int_0^t u^{\infty}(\alpha)d\alpha
   +\int_0^{1-\int_0^t\lambda(\xi_1^{\infty}(\theta),W^{\infty}(\theta))d\theta}
     \rho_0(\beta)d\beta,
   \quad  \forall t \in [0,\delta].
 \ee

We know from the  proof of Theorem \ref{thm-sol} that $W=F(W)$ has a
unique fixed point in $\Omega_{\delta,\wt M}$ (replacing $M$ by $\wt
M$ in \eqref{Omega}). This implies from \eqref{ov-wInf} and
\eqref{wInf} that $W^{\infty}(t)\equiv \ov W^{\infty}(t)$ on $
[0,\delta]$, hence for any fixed $t\in [0,\delta]$,
$\xi_1^{\infty}(s)\equiv \ov \xi_1^{\infty}(s),
\xi_2^{\infty}(s)\equiv \ov \xi_2^{\infty}(s)$ on $ [0,t]$.
Moreover, with the help of \eqref{delta}, there exists $\delta_0>0$
independent of $\tau\in (0,T)$ such that if $W^{\infty}(\tau)=\ov
W^{\infty}(\tau)$ then $W^{\infty}(t)\equiv\ov W^{\infty}(t)$ on
$[\tau,\tau+\delta_0]\cap[\tau,T]$. Hence we have $W^n \rightarrow
W^{\infty}$ in $C^0([0,T])$ and furthermore $\xi_1^n \rightarrow
\xi_1^{\infty}$ uniformly and $\xi_2^n \rightarrow \xi_2^{\infty}$
in $C^1([0,T])$.

Next we prove that $y^n(t)=\rho^n(t,1)\lambda(1,W^n(t))$ converges
to $y^{\infty}(t)=\rho^{\infty}(t,1)\lambda(1,W^{\infty}(t))$ weakly
in $L^p(0,T)$ for $p\in [1,\infty)$. By \eqref{un-bound},
$\{y^n\}_{n=1}^{\infty}$ is bounded in $L^p(0,T)$. Hence, it
suffices to prove that for any $g\in C^1([0,T])$,
 \be\label{yn-yInf}
 \lim_{n\rightarrow \infty}\int_0^T(y^n(t)-y^{\infty}(t))g(t)dt=0.
 \ee

\emph{Case 1. \ $\xi_2^{\infty}(T)<1$.}
 \begin{align*}
 &\quad \Big|\int_0^T(y^n(t)-y^{\infty}(t))g(t)dt\Big|
     \nonumber \\
  & =\Big | \int_0^T  \Big(\rho_0(\xi_1^n(0)))\lambda(1,W^n(t))
       \, e^{- \int_0^t \lambda_x(\xi_1^n(\theta),W^n(\theta)) d\theta}
       \nonumber \\
  &\qquad \qquad \qquad  -\rho_0(\xi_1^{\infty}(0))\lambda(1,W^{\infty}(t))
       \, e^{- \int_0^t \lambda_x(\xi_1^{\infty}(\theta),W^{\infty}(\theta))
       d\theta}\Big) g(t) dt \Big|.
  \end{align*}
Let $\{\rho_0^k\}_{k=1}^{\infty} \subset C^1([0,1])$ be such that
$\rho_0^k \rightarrow \rho_0$ in $L^p(0,1)$, then by \eqref{rho},
\eqref{wn-bound} and Lemma \ref{jacobi2}, we have
 \begin{align}\label{yn-yInf-1}
 &\quad \Big|\int_0^T(y^n(t)-y^{\infty}(t))g(t)dt\Big|
     \nonumber \\
  & \leq \int_0^T  |\rho_0^k(\xi_1^n(0))-\rho_0(\xi_1^n(0))| \lambda(1,W^n(t))
       \, e^{- \int_0^t \lambda_x(\xi_1^n(\theta),W^n(\theta))
       d\theta} |g(t)| dt
      \nonumber \\
  & \quad    +\int_0^T  |\rho_0^k(\xi_1^{\infty}(0))-\rho_0(\xi_1^{\infty}(0))| \lambda(1,W^{\infty}(t))
       \, e^{- \int_0^t \lambda_x(\xi_1^{\infty}(\theta),W^{\infty}(\theta))
       d\theta} |g(t)| dt
      \nonumber \\
   &  \quad   +\int_0^T  \Big|\rho_0^k(\xi_1^n(0))\lambda(1,W^n(t))
       \, e^{- \int_0^t \lambda_x(\xi_1^n(\theta),W^n(\theta))
       d\theta}
       \nonumber \\
   & \qquad \qquad \qquad
         -\rho_0^k(\xi_1^{\infty}(0)) \lambda(1,W^{\infty}(t))
       \, e^{- \int_0^t \lambda_x(\xi_1^{\infty}(\theta),W^{\infty}(\theta))
       d\theta}\Big| |g(t)| dt
       \nonumber \\
  & \leq C \|\rho_0^k-\rho_0\|_{L^p(0,1)} +C_k \|W^n-W^{\infty}\|_{C^0([0,T])}
           + C_k \int_0^T |\xi_1^n(0)-\xi_1^{\infty}(0)| |g(t)|dt
        \nonumber \\
  &\quad   +C_k \int_0^T  \Big| \int_0^t (\lambda_x(\xi_1^n(\theta),W^n(\theta))
           -\lambda_x(\xi_1^{\infty}(\theta),W^{\infty}(\theta))) d \theta \Big| |g(t)|dt
  \end{align}
where $C_k$ is a constant depending on $\rho_0^k$. Noting the fact
that $\rho_0^k \rightarrow \rho_0$ in $L^p(0,1)$ and
$\xi_1^n\rightarrow \xi_1^{\infty}$, $W^n\rightarrow W^{\infty}$
uniformly, by Lebesgue dominated convergence theorem, we let $k$
large enough first and then $n$ large enough, so that the right hand
side of \eqref{yn-yInf-1} can be arbitrarily small. This concludes
\eqref{yn-yInf} for the case of $\xi_2^{\infty}(T)<1$.

\emph{Case 2. \ $\xi_2^{\infty}(T)=1$, i.e.,
$T=(\xi_2^{\infty})^{-1}(1)$.} For every $\tau\in [0,T)$, we have
 \begin{align}\label{yn-yInf-2}
  & \quad \Big|\int_0^T  (y^n(t)-y^{\infty}(t))g(t)dt\Big|
  \nonumber\\
 & =  \Big|\int_0^{\tau}(y^n(t)-y^{\infty}(t))g(t)dt
     +\int_{\tau}^T (y^n(t)-y^{\infty}(t))g(t)dt\Big|
  \nonumber\\
 & \leq  \Big|\int_0^{\tau}(y^n(t)-y^{\infty}(t))g(t)dt\Big|
    + C (T-\tau)^{\frac 12}.
 \end{align}
Since it is known from Case 1 that for every $\tau\in [0,T)$,
$|\int_0^{\tau}(y^n(t) -y^{\infty}(t)) g(t) dt | \rightarrow 0$ as
$n \rightarrow\infty$, one has  \eqref{yn-yInf} for
$T=(\xi_2^{\infty})^{-1}(1)$ by letting $\tau \rightarrow T$ in
\eqref{yn-yInf-2} .

\emph{Case 3. \ $\xi_2^{\infty}(T)>1$.} Now we have
 \beq
 \int_0^T(y^n(t)-y^{\infty}(t))g(t)dt
 =\Big(\int_0^{(\xi_2^{\infty})^{-1}(1)}
 +\int_{(\xi_2^{\infty})^{-1}(1)}^T\Big) (y^n(t)-y^{\infty}(t))g(t) dt.
 \eeq
Using the result in Case 2, we need only to prove that
$|\int_{(\xi_2^{\infty})^{-1}(1)}^T (y^n(t)-y^{\infty}(t))g(t) dt|
\rightarrow 0.$

For any fixed $\alpha\in [0,T]$, we define
$\wt\xi_1^n,\wt\xi_1^{\infty}$ by
 \begin{align*}
  & \f{d\wt\xi_1^n}{ds}=\lambda(\wt\xi_1^n(s),W^n(s)),\quad
     \wt\xi_1^n(\alpha)=0,
  \\ & \f{d\wt\xi_1^{\infty}}{ds}=\lambda(\wt\xi_1^{\infty}(s),W^{\infty}(s)),\quad
     \wt\xi_1^{\infty}(\alpha)=0,
 \end{align*}
And we know from $W^n \rightarrow W^{\infty}$ in $C^0([0,T])$ that
$\wt\xi_1^n \rightarrow \wt\xi_1^{\infty}$ and $(\wt\xi_1^n)^{-1}
\rightarrow (\wt\xi_1^{\infty})^{-1}$ uniformly.

Then we get from \eqref{rho} that
 \begin{align*}
  &\quad  \Big|\int_{(\xi_2^{\infty})^{-1}(1)}^T (y^n(t)-y^{\infty}(t))g(t)dt\Big|
    \nonumber\\
 &= \Big| \int_{(\xi_2^{\infty})^{-1}(1)}^T
      (\rho^n(t,1)\lambda(1,W^n(t))-\rho^{\infty}(t,1)\lambda(1,W^{\infty}(t)))  g(t)dt\Big|
  \nonumber\\
 &= \Big| \int_{(\xi_2^{\infty})^{-1}(1)}^T
      \Big( \f{u^n(\alpha^n)\lambda(1,W^n(t))}{\lambda(0,W^n(\alpha^n))}
            \, e^{-\int_{\alpha^{^n}}^t
            \lambda_x(\xi_1^n(\theta),W^n(\theta)) d\theta}
    \nonumber\\
  & \qquad \qquad \qquad
      -  \f{u^{\infty}(\alpha^{\infty})\lambda(1,W^{\infty}(t))}{\lambda(0,W^{\infty}(\alpha^{\infty}))}
            \, e^{-\int_{\alpha^{^\infty}}^t   \lambda_x(\xi_1^{\infty}(\theta),W^{\infty}(\theta)) d\theta}
            \Big)  g(t)dt\Big|,
\end{align*}
where $\alpha^n=(\xi_1^n)^{-1}(0), \alpha^{\infty}=
(\xi_1^{\infty})^{-1}(0)$ are determined by
  \be \label{alpha n}
    1-\int_{\alpha^n}^t \lambda(\xi_1^n(\theta),W^n(\theta)) d\theta
    =  1-\int_{\alpha^{\infty}}^t
     \lambda(\xi_1^{\infty}(\theta),W^{\infty}(\theta))
    d\theta  =0,   \ee
and thus
   \begin{align}  \label{d-alpn-dt}
     &\f{d\alpha^n}{dt}= \f{\lambda(1,W^n(t))}{\lambda(0,W^n(\alpha^n))}
          \,e^{-\int_{\alpha^n}^t\lambda_x(\xi_1^n(\theta),W^n(\theta)) d\theta},
      \\\label{d-alph-Inf-dt}
     & \f{d\alpha^{\infty}}{dt} =\f{\lambda(1,W^{\infty}(t))}{\lambda(0,W^{\infty}(\alpha^{\infty}))}
         \,e^{-\int_{\alpha^{\infty}}^t\lambda_x(\xi_1^{\infty}(\theta),W^{\infty}(\theta)) d\theta}.
      \end{align}

Let  $(\tau^n,\tau^{\infty})$ be the value of
$(\alpha^n,\alpha^{\infty})$ when $t=T$ in \eqref{alpha n}, and let
$(\eta^n,\eta^{\infty})$ be the value of
$(\alpha^n,\alpha^{\infty})$ when $t=(\xi_2^{\infty})^{-1}(1)$ in
\eqref{alpha n}. Since $W^n\rightarrow W^{\infty}$, $\xi_1^n
\rightarrow \xi_1^{\infty}$ uniformly, then $\tau^n\rightarrow
\tau^{\infty}$, $\eta^n\rightarrow \eta^{\infty}$ and
$\alpha^n\rightarrow \alpha^{\infty}$ uniformly. Hence, by
\eqref{d-alpn-dt}-\eqref{d-alph-Inf-dt} and using the facts that
$u^n \xrightharpoonup[]{*} u^{\infty}$ and $\wt\xi_1^n \rightarrow
\wt\xi_1^{\infty}, (\wt\xi_1^n)^{-1} \rightarrow
(\wt\xi_1^{\infty})^{-1}$ uniformly, we  obtain
\begin{align*}
  &\quad  \Big|\int_{(\xi_2^{\infty})^{-1}(1)}^T (y^n(t)-y^{\infty}(t))g(t)dt\Big|
    \nonumber\\
  &= \Big| \int_{\eta^n}^{\tau^n}   u^n(\alpha^n) g(t)d  \alpha^n
         - \int_{\eta^{\infty}}^{\tau^{\infty}}  u^{\infty}(\alpha^{\infty}) g(t)d \alpha^{\infty}  \Big|
    \nonumber\\
  &\leq  C |\eta^n-\eta^{\infty}|+C|\tau^n-\tau^{\infty}|
      + \int_{\eta^{\infty}}^{\tau^{\infty}}
          \Big| u^n(\alpha)  g((\wt \xi_1^n)^{-1}(1))
             - u^{\infty}(\alpha) g((\wt \xi_1^{\infty})^{-1}(1))\Big| d \alpha
   \nonumber\\
  &\leq  C |\eta^n-\eta^{\infty}|+C|\tau^n-\tau^{\infty}|
    + C \int_{\eta^{\infty}}^{\tau^{\infty}}
              | g((\wt \xi_1^n)^{-1}(1))- g((\wt \xi_1^{\infty})^{-1}(1)) | d \alpha
    \nonumber\\
   &\quad +\int_{\eta^{\infty}}^{\tau^{\infty}}
       | u^n(\alpha)- u^{\infty}(\alpha)|
   |g((\wt \xi_1^{\infty})^{-1}(1))| d \alpha
     \nonumber\\
    & \quad \longrightarrow 0, \qquad \text{as}\quad  n \rightarrow \infty.
\end{align*}
This concludes the proof of \eqref{yn-yInf} for the case
$\xi_2^{\infty}(T)>1$.

As a result,
 \begin{align*}
 J_p(u^{\infty})&=\|u^{\infty}\|_{L^{\infty}(0,T)}^2
    + \|y^{\infty}-y_d\|_{L^p(0,T)}^2
  \nonumber\\
 &\leq  \liminf_{n\rightarrow \infty} \|u^n\|_{L^{\infty}(0,T)}^2
    + \liminf_{n\rightarrow \infty} \|y^n-y_d\|_{L^p(0,T)}^2
  \nonumber\\
 & \leq  \liminf_{n\rightarrow \infty} J_p(u^n)=\lim_{n\rightarrow \infty} J_p(u^n)
    =\inf_{u\in L^{\infty}_{+}(0,T)}J_p(u).
 \end{align*}
This shows $u_{\infty}$ is a minimizer of $J_p(u)$ in
$L^{\infty}_{+}(0,T)$, and it proves also that $y^n \rightarrow
y^{\infty}$ strongly in $L^p(0,T)$.
\end{proof}


\appendix

\section{Appendix}

\subsection{Basic Lemmas}

The following two lemmas are used to prove the existence of weak
solution to Cauchy problem \eqref{eq}, \eqref{eq-IC} and
\eqref{eq-BC}, when changing variables in certain integrals (see
Section A.2). We recall some assumptions that are given:
$\lambda>0$, $\lambda\in C^1([0,1]\times [0,\infty))$, $W\in
\Omega_{\delta,M}$ (see \eqref{Omega} for definition).

  \begin{lem}\label{jacobi1}
  Let $\xi_3$ be the characteristic (see \eqref{xi3} for definition)
  which passes through the point $(t,x)$ and
  intersects the $t$-axis at the point $(\alpha,0)$.
  Then we have
  \be  \f{d\alpha}{dx}=-\f{1}{\lambda(0,W(\alpha))}\,
   e^{-\int_{\alpha}^t \lambda_x(\xi_3(\theta),W(\theta))d\theta}.
  \ee
  \end{lem}

 \begin{lem}\label{jacobi2}
  Let $\xi_4$ be the characteristic (see \eqref{xi4} for definition)
  which passes through the point $(t,x)$
  and intersects the $x$-axis at the point $(0,\beta)$.
  Then we have
  \be  \f{d\beta}{dx}=
   e^{-\int_0^t \lambda_x(\xi_4(\theta),W(\theta))d\theta}.
  \ee
  \end{lem}

The proofs of Lemma \ref{jacobi1} and Lemma \ref{jacobi2} are
trivial and they can be found in \cite{LiYuBook}.

The next lemma is useful to prove a uniqueness result (Section 2,
Lemma \ref{uni-sol}) for Cauchy problem \eqref{eq}, \eqref{eq-IC}
and \eqref{eq-BC}.

 \begin{lem}\label{lem-test-function}
If $\rho\in C^0([0,T];L^1(0,1)) \cap L^{\infty}((0,T)\times (0,1))$
is a weak solution to Cauchy problem \eqref{eq}, \eqref{eq-IC} and
\eqref{eq-BC}, then for every $t\in[0,T]$ and every $\varphi\in
C^1([0,t]\times[0,1])$ such that
 \beq
  \varphi(\tau,1)= 0,\quad \forall \tau\in[0,t],
 \eeq
one has
 \begin{multline*}
  \int_0^{t} \int_0^1 \rho(\tau,x)(\varphi_{\tau}(\tau,x)
   +\lambda(x, W(\tau))\varphi_x(\tau,x))dxd\tau
   +\int_0^{t}u(\tau)\varphi(\tau,0)d\tau
   \\
 -\int_0^1 \rho(t,x)\varphi(t,x)dx
    +\int_0^1\rho_0(x)\varphi(0,x)dx =0.
 \end{multline*}
\end{lem}

The proof of Lemma \ref{lem-test-function} is the same as the proof
of Lemma 2.2 in \cite[Section 2.2]{CKWang}.

\subsection{Proof of Lemma \ref{rho-sol}}
First we recall some notations $M, \ov \lambda(M),
\|\lambda\|_{C^0}, \|\lambda_x\|_{C^0},\|\lambda_W\|_{C^0}$ which
are defined in Section 2. And the characteristic curves
$\xi_1,\xi_2,\xi_3,\xi_4$ are defined by \eqref{xi1},
\eqref{xi2},\eqref{xi3}, \eqref{xi4}, respectively.

By definition \eqref{rho-loc}, it is easy to see that $\rho\in
L^{\infty} ((0,\delta)\times (0,1))$ and that estimates
\eqref{wt-bound-loc} and \eqref{rho-bound-loc} hold. Next we prove
that the function defined by \eqref{rho-loc} belongs to
$C^0([0,\delta];L^p(0,1))$ for all $p\in [1,\infty)$, i.e., for
every $\wt t,t\in [0,\delta]$ with $\wt t\geq t$, we need to prove
 \beq  \|\rho(\wt t,\cdot)-\rho(t,\cdot)\|_{L^p(0,1)}\rightarrow 0,
    \quad \text{as} \quad |\wt t-t|\rightarrow 0
   \eeq
for all $p\in [1,\infty)$. In order to do that, we estimate
$\int_0^{\xi_2(t)} |\rho(\wt t,x)-\rho(t,x)|^p dx$,
$\int_{\xi_2(\,\wt t\,)}^1 |\rho(\wt t,x)-\rho(t,x)|^p dx$ and
$\int_{\xi_2(t)}^{\xi_2(\,\wt t\,)} |\rho(\wt t,x)-\rho(t,x)|^p dx$,
successively.

For almost every $x\in [0,\xi_2(t)]$, by \eqref{rho-loc} and noting
$W\in \Omega_{\delta,M}$ (see \eqref{Omega} for definition), we know
(see Fig \ref{Fig10})
 \begin{align*}
     &\quad  |\rho(\wt t,x)-\rho(t,x)|
                \\
    &= \Big|\f{u(\wt \alpha)}{\lambda(0,W(\wt \alpha))}
            \, e^{-\int_{\wt \alpha}^{\wt t} \lambda_x(\wt \xi_3(\theta),W(\theta)) d\theta}
        - \f{u(\alpha)}{\lambda(0,W(\alpha))}
            \, e^{-\int_{\alpha}^t \lambda_x(\xi_3(\theta),W(\theta)) d\theta} \Big|
                \\
    &\leq \f{|u(\wt \alpha)-u(\alpha)|}{\lambda(0,W(\wt \alpha))}
         \, e^{-\int_{\wt \alpha}^{\wt t} \lambda_x(\wt \xi_3(\theta),W(\theta)) d\theta}
                \\
    &\quad + \Big|\f{u(\alpha)}{\lambda(0,W(\wt \alpha))}
          -\f{u( \alpha)}{\lambda(0,W(\alpha))}\Big|
            \, e^{-\int_{\wt \alpha}^{\wt t} \lambda_x(\wt \xi_3(\theta),W(\theta)) d\theta}
                \\
    &\quad +\f{u(\alpha)}{\lambda(0,W(\alpha))}
            \, e^{-\int_{\wt \alpha}^{\wt t} \lambda_x(\wt \xi_3(\theta),W(\theta)) d\theta}
          - e^{-\int_{\alpha}^t \lambda_x(\xi_3(\theta),W(\theta)) d\theta} \Big|
               \\
    &\leq C |u(\wt \alpha)-u(\alpha)|+C|u( \alpha)||W(\wt \alpha)-W(\alpha)|
               \\
     &\quad +C|u( \alpha)| \Big| \int_{\wt \alpha}^{\wt t} \lambda_x(\wt \xi_3(\theta),W(\theta)) d\theta
                 -\int_{\alpha}^t \lambda_x(\xi_3(\theta),W(\theta)) d\theta \Big|,
  \end{align*}
where $\wt \xi_3$ denotes the characteristic curve passing through
$(\wt t,x)$:
 \beq   \f{d \wt \xi_3}{ds}=\lambda(\wt \xi_3(s),W(s)),
   \quad \wt \xi_3(\,\wt t\,)=x
   \eeq
and it intersects $t$-axis at $(\wt \alpha,0)$: $\wt \xi_3(\wt
\alpha)=0$.  Here  and hereafter in Appendix, we denote by $C$
various constants which do not depend on $x,t,\wt t$.

For the given $u\in L^{\infty}(0,T)$, we let $\{u_n\}_{n=1}^{\infty}
\subset C^1([0,T])$ be such that $u_n \rightarrow u$ in $L^p(0,T)$.
And for the given $\lambda \in C^1([0,1]\times [0,\infty))$, we let
$\{v_n\}_{n=1}^{\infty} \subset C^1([0,1]\times [0,M])$ be such that
$v_n \rightarrow \lambda_x$ in $C^0([0,1]\times [0,M])$. Using the
sequences $\{u_n\}_{n=1}^{\infty}$, $\{v_n\}_{n=1}^{\infty}$ and
noting \eqref{w't-bound-loc}, we have for almost every $x\in
[0,\xi_2(t)]$ that
 \begin{align}\label{rho-cont-1}
   &\quad  |\rho(\wt t,x)-\rho(t,x)|
                \nonumber\\
     &\leq C |u(\wt \alpha)-u(\alpha)|+C|u( \alpha)||W(\wt \alpha)-W(\alpha)|
              \nonumber\\
     &\quad +C|u( \alpha)| \Big| \int_{\wt \alpha}^{\wt t} \lambda_x(\wt \xi_3(\theta),W(\theta)) d\theta
                 -\int_{\alpha}^t \lambda_x(\xi_3(\theta),W(\theta)) d\theta \Big|
            \nonumber\\
     &\leq C |u^n(\wt \alpha)-u(\wt \alpha)|+  C|u^n( \alpha)-u( \alpha)|
          + C|u^n(\wt \alpha)-u^n(\alpha)| +C|u( \alpha)||W(\wt \alpha)-W(\alpha)|
              \nonumber\\
     &\quad +C|u( \alpha)|   \int_{\wt \alpha}^{\wt t}
           |v_n(\wt \xi_3(\theta),W(\theta))-\lambda_x(\wt \xi_3(\theta),W(\theta))| d\theta
               \nonumber\\
     &\quad +C|u( \alpha)| \int_{\alpha}^{t}
           |v_n(\xi_3(\theta),W(\theta))-\lambda_x(\xi_3(\theta),W(\theta))| d\theta
                \nonumber\\
     &\quad +C|u( \alpha)| \Big| \int_{\wt \alpha}^{\wt t} v_n(\wt \xi_3(\theta),W(\theta)) d\theta
                 -\int_{\alpha}^t v_n(\xi_3(\theta),W(\theta)) d\theta \Big|
                \nonumber\\
     &\leq C |u^n(\wt \alpha)-u(\wt \alpha)|+  C|u^n( \alpha)-u( \alpha)|
          +C  |u( \alpha)| \|v_n-\lambda_x\|_{C^0([0,1]\times [0,M])} + C_n|\wt \alpha - \alpha|
               \nonumber \\
    & \quad  +C_n |u( \alpha)| |\wt t-t|+C_n |u( \alpha)| |\wt \alpha-\alpha|
          + C_n |u( \alpha)| \|\wt \xi_3-\xi_3\|_{C^0([\wt \alpha,t])}.
   \end{align}
Here  and hereafter, we denote by $C_n$ various constants which do
not depend on $x,t,\wt t$ but may depend on $n$ (the index of the
corresponding approximating sequences, e.g.
$\{u_n\}_{n=1}^{\infty}$, $\{v_n\}_{n=1}^{\infty}$ and so on).

By the definitions of $\wt \xi_3, \xi_3$ and \eqref{delta}, we
derive
 \begin{align*} |\wt \xi_3(s)-\xi_3(s)|
   &=\Big |\int_s^{\wt t} \lambda(\wt \xi_3(\theta),W(\theta)) d\theta
      - \int_s^t\lambda(\xi_3(\theta),W(\theta))d\theta \Big |
      \\
   & \leq C |\wt t-t|+ \delta \|\lambda_x\|_{C^0}
     \|\wt \xi_3-\xi_3\|_{C^0([\wt \alpha,t])},
    \quad  \forall s\in [\wt \alpha,t],
  \end{align*}
and furthermore
 \be \label{wt-xi3-xi3} \|\wt \xi_3-\xi_3\|_{C^0([\wt \alpha,t])}
  \leq C |\wt t-t|.  \ee
Meanwhile, from the fact that
 \beq
  \wt \xi_3(\,\wt t\,)=\int_{\wt \alpha}^{\wt t}
  \lambda(\wt\xi_3(\theta),W(\theta))d\theta =x
  =\xi_3(t)= \int_{\alpha}^t
  \lambda(\xi_3(\theta),W(\theta))d\theta,
 \eeq
and  definition \eqref{ov-lam} of $\ov \lambda(M)$, we get
  \begin{align} \label{wt-alp-alp}
     |\wt \alpha-\alpha|
  &\leq \f1 {\ov \lambda(M)} \Big| \int_{\alpha}^{\wt
     \alpha}\lambda(\xi_3(\theta),W(\theta))d\theta \Big|
    \nonumber \\
   & = \f1 {\ov \lambda(M)}
      \Big|  \int_{\wt \alpha}^{t} (\lambda(\wt \xi_3(\theta),W(\theta))
               -\lambda(\xi_3(\theta),W(\theta)) ) d\theta
            +\int_t^{\wt t}\lambda(\wt \xi_3(\theta),W(\theta)) d\theta       \Big|
    \nonumber \\
   & \leq C \|\wt \xi_3-\xi_3\|_{C^0([\wt \alpha,t])}+ C|\wt t-t|
    \nonumber \\
   & \leq C |\wt t-t|.
   \end{align}
Therefore, we get easily from \eqref{rho-cont-1}, \eqref{wt-xi3-xi3}
\eqref{wt-alp-alp} and Lemma \ref{jacobi1}, that
  \begin{align}\label{rho-cont-2}
    & \quad \int_0^{\xi_2(t)}|\rho(\wt t,x)-\rho(t,x)|^p dx
     \nonumber \\
     & \leq  C \|u_n-u\|_{L^p(0,T)}^p +C\|v_n-\lambda_x\|^p_{C^0([0,1]\times[0,M])}
        + C_n |\wt t-t|^p.
       \end{align}
For almost every $x\in [\xi_2(\,\wt t\,),1]$, by definition
\eqref{rho-loc} of $\rho$, we have (see Fig \ref{Fig11})
 \begin{align*}
   &\quad |\rho(\wt t,x)-\rho(t,x)|
        \\
   & = \Big|\rho_0(\wt \beta)\, e^{-\int_0^{\wt t} \lambda_x(\wt \xi_4(\theta),W(\theta)) d\theta}
              - \rho_0(\beta) \, e^{-\int_0^t \lambda_x(\xi_4(\theta),W(\theta)) d\theta} \Big|
        \\
   & \leq |\rho_0(\wt \beta)-\rho_0(\beta)|
            \, e^{-\int_0^{\wt t} \lambda_x(\wt \xi_4(\theta),W(\theta)) d\theta}
        + |\rho_0(\beta)| \Big | e^{-\int_0^{\wt t} \lambda_x(\wt \xi_4(\theta),W(\theta)) d\theta}
              -e^{-\int_0^t \lambda_x(\xi_4(\theta),W(\theta))   d\theta}\Big|
        \\
   &\leq C |\rho_0(\wt \beta)-\rho_0(\beta)|
       + C |\rho_0(\beta)| \Big| \int_0^{\wt t} \lambda_x(\wt \xi_4(\theta),W(\theta)) d\theta
           -\int_0^t \lambda_x(\xi_4(\theta),W(\theta))   d\theta\Big|
 \end{align*}
where the characteristic curve $\wt \xi_4$ passing through $(\wt
t,x)$ is defined by
 \beq   \f{d \wt \xi_4}{ds}=\lambda(\wt \xi_4(s),W(s)),
   \quad \wt \xi_4(\,\wt t\,)=x
   \eeq
and it intersects $x$-axis at $(0,\wt \beta)$: $\wt \xi_4(0)= \wt
\beta$.

For the given $\rho_0\in L^{\infty}(0,1)$, we let
$\{\rho_0^n\}_{n=1}^{\infty} \subset C^1([0,1])$ be such that
$\rho_0^n \rightarrow \rho_0$ in $L^p(0,1)$. Using  sequences
$\{\rho_0^n\}_{n=1}^{\infty}$ and $\{v_n\}_{n=1}^{\infty}$, we
obtain for almost every $x \in [\xi_2(\,\wt t\,),1]$ that
 \begin{align}\label{rho-cont-3}
   & \qquad   |\rho(\wt t,x)-\rho(t,x)|
    \nonumber \\
   & \leq C |\rho_0^n(\wt \beta)-\rho_0(\wt \beta)|+  C|\rho_0^n(\beta)-\rho_0(\beta)|
       + C  |\rho_0^n(\wt\beta)-\rho_0^n(\beta)|
       +C |\rho_0(\beta)| \|v_n-\lambda_x\|_{C^0([0,1]\times [0,M])}
    \nonumber \\
   & \quad + C |\rho_0(\beta)|  \Big | \int_0^{\wt t} v_n(\wt \xi_4(\theta),W(\theta))   d\theta
     -\int_0^t v_n(\xi_4(\theta),W(\theta))  d\theta\Big|
    \nonumber \\
   & \leq C |\rho_0^n(\wt \beta)-\rho_0(\wt \beta)|+  C|\rho_0^n(\beta)-\rho_0(\beta)|
        +C |\rho_0(\beta)| \|v_n-\lambda_x\|_{C^0([0,1]\times [0,M])}
    \nonumber \\
   & \quad   + C_n  |\wt\beta-\beta| +C_n |\rho_0(\beta)| |\wt t-t|
         + C_n |\rho_0(\beta)| \|\wt  \xi_4-\xi_4\|_{C^0([0,t])}
  \end{align}
By the definitions of $\wt \xi_4, \xi_4$, we have
 \be \label{wt-xi4-xi4} \|\wt \xi_4-\xi_4\|_{C^0([0,t])}
  \leq C |\wt t-t|,  \ee
which is similar to \eqref{wt-xi3-xi3}. In particular,
 \be \label{wt-beta-beta}|\wt \beta-\beta|= |\wt \xi_4(0)-\xi_4(0)|
  \leq C |\wt t-t|.  \ee
Therefore, we get easily from \eqref{rho-cont-3},
\eqref{wt-xi4-xi4}, \eqref{wt-beta-beta}  and Lemma \ref{jacobi2}
that
  \begin{align}\label{rho-cont-4}
    & \quad \int_{\xi_2(\,\wt t\,)}^1|\rho(\wt t,x)-\rho(t,x)|^p dx
      \nonumber \\
    & \leq  C \|\rho_0^n-\rho_0\|_{L^p(0,1)}^p +C\|v_n-\lambda_x\|^p_{C^0([0,1]\times[0,M])}
       + C_n |\wt t-t|^p
            \end{align}

Finally, we turn to estimate $ \int_{\xi_2(t)}^{\xi_2(\,\wt
t\,)}|\rho(\wt t,x)-\rho(t,x)|^p dx$  (see Fig \ref{Fig12}).
Obviously, we have
  \be \label{rho-cont-5}
     \int_{\xi_2(t)}^{\xi_2(\,\wt t\,)}|\rho(\wt t,x)-\rho(t,x)|^p dx
      \leq \int_{\xi_2(t)}^x |\rho(\wt t,x)-\rho(t,x)|^p dx
      +\int_x^{\xi_2(\,\wt t\,)}|\rho(\wt t,x)-\rho(t,x)|^p dx.
      \ee
Then similar to estimates \eqref{rho-cont-2} and \eqref{rho-cont-4},
we get easily from \eqref{rho-cont-5} that
 \begin{align} \label{rho-cont-6}
   & \quad \int_{\xi_2(t)}^{\xi_2(\,\wt t\,)}|\rho(\wt t,x)-\rho(t,x)|^p dx
     \nonumber\\
   &\leq  C \|u_n-u\|_{L^p(0,T)}^p+ C \|\rho_0^n-\rho_0\|_{L^p(0,1)}^p
        +C\|v_n -\lambda_x\|^p_{C^0([0,1]\times[0,M])}
       + C_n |\wt t-t|^p.
   \end{align}

Summarizing  estimates \eqref{rho-cont-2}, \eqref{rho-cont-4} and
\eqref{rho-cont-6}, we find that for every $p\in [1,\infty)$,
 \begin{align} \label{rho-cont-7}
   &\quad \|\rho(\wt t,\cdot)-\rho(t,\cdot)\|_{L^p(0,1)}^p
     \nonumber\\
   &\leq  C \|u_n-u\|_{L^p(0,T)}^p+ C \|\rho_0^n-\rho_0\|_{L^p(0,1)}^p
        +C\|v_n -\lambda_x\|^p_{C^0([0,1]\times[0,M])}
       + C_n |\wt t-t|^p .
   \end{align}
Therefore by Lebesque dominated convergence theorem, letting $n$
large enough and then $|\wt t-t|$ small enough, the right hand side
of \eqref{rho-cont-7} can be arbitrarily small. This proves that the
function $\rho$ defined by \eqref{rho-loc} belongs to
$C^0([0,\delta];L^p(0,1))$ for all $p\in [1,\infty)$.

Finally, we prove that $\rho$ defined by \eqref{rho-loc} is indeed a
weak solution to Cauchy problem \eqref{eq}, \eqref{eq-IC} and
\eqref{eq-BC}. Let $\tau \in [0,\delta]$. For any $\varphi\in
C^1([0,\tau]\times [0,1])$ with $\varphi(\tau,x)\equiv 0$ and
$\varphi(t,1)\equiv 0$, we have
 \begin{align*}
  A& :=\int_0^{\tau} \int_0^1 \rho(t,x)(\varphi_t(t,x)
        +\lambda(x,W(t))\varphi_x(t,x)) dx dt
   \\
 & =\int_0^{\tau} \int_0^{\xi_2(t)} \f{u(\alpha)}{\lambda(0,W(\alpha))}
      \, e^{-\int_{\alpha}^t \lambda_x(\xi_3(\theta),W(\theta)) d\theta}
         \cdot (\varphi_t(t,x)+\lambda(x,W(t))\varphi_x(t,x)) dx dt
 \nonumber\\
 & \quad +\int_0^{\tau}\int_{\xi_2(t)}^1
         \rho_0(\beta)\,  e^{-\int_0^t \lambda_x(\xi_4(\theta),W(\theta)) d\theta }
        \cdot (\varphi_t(t,x) +\lambda(x,W(t))\varphi_x(t,x)) dx dt
  \end{align*}
By Lemma \ref{jacobi1} and Lemma \ref{jacobi2}, we obtain
    \begin{align*}
 A &=\int_0^{\tau} \int_0^t u(\alpha)(\varphi_t(t,\xi_3(t))
     +\lambda(\xi_3(t),W(t))\varphi_x(t,\xi_3(t)) d\alpha dt
  \nonumber\\
 & \quad +\int_0^{\tau}\int_0^{1-\int_0^t \lambda(\xi_1(\theta),W(\theta))d\theta}
     \rho_0(\beta) (\varphi_t(t,\xi_4(t)) +\lambda(\xi_4(t),W(t))\varphi_x(t,\xi_4(t))) d\beta dt
 \nonumber\\
 &=\int_0^{\tau} \int_0^t u(\alpha) \f{d\varphi (t,\xi_3(t))}{dt} d\alpha dt
     +\int_0^{\tau}\int_0^{1-\int_0^t \lambda(\xi_1(\theta),W(\theta))d\theta}
      \rho_0(\beta) \f{d\varphi(t,\xi_4(t))}{dt}  d\beta dt.
  \end{align*}
Hence by changing the order of integral, we get
 $$ A =\int_0^{\tau} \int_{\alpha}^{\tau}
     u(\alpha)\frac{d\varphi(t,\xi_3(t))}{dt} dt d\alpha
      +\Big(\int_0^{f(\tau)}\int_0^{\tau }
        +\int_{f(\tau)}^1\int_0^{f^{-1}(\beta)}\Big)
         \rho_0(\beta)\frac{d\varphi(t,\xi_4(t))}{dt} dt d\beta
 $$
where
 \beq f(t):=1-\int_0^t \lambda(\xi_1(\theta),W(\theta))d\theta
 \eeq
represents the coordinate that the characteristic curve $\xi_1$
intersects with $x$-axis and $f^{-1}(\beta)$ represents the time
when the characteristic curve staring from $(0,\beta)$ arrives at
the boundary $x=1$. Consequently, for any $\beta \in [f(\tau),1]$,
$\xi_1$ and $\xi_4$ are identical to each other since they pass
through the same point $(0,\beta)$,  so we get immediately that
 \beq \xi_4(f^{-1}(\beta))=\xi_1(f^{-1}(\beta))=1.
 \eeq
and finally that
  \begin{align*}
  A &=\int_0^{\tau}  u(\alpha) (\varphi(\tau,\xi_3(\tau))- \varphi(\alpha,\xi_3(\alpha))) d\alpha
      +  \int_0^{f(\tau)}   \rho_0(\beta)
          (\varphi(\tau,\xi_4(\tau))- \varphi(0,\xi_4(0)))  d\beta
   \\
   &\quad   +  \int_{f(\tau)}^1   \rho_0(\beta)
         (\varphi(f^{-1}(\beta),\xi_4(f^{-1}(\beta)))- \varphi(0,\xi_4(0)))   d\beta
   \\
  &= -\int_0^{\tau}  u(\alpha) \varphi(\alpha,0) d\alpha
      - \int_0^{f(\tau)}   \rho_0(\beta) \varphi(0,\beta)  d\beta
      -  \int_{f(\tau)}^1   \rho_0(\beta) \varphi(0,\beta)  d\beta
    \\
  &= -\int_0^{\tau}  u(\alpha) \varphi(\alpha,0) d\alpha
      - \int_0^1   \rho_0(\beta) \varphi(0,\beta)  d\beta.
  \end{align*}

This proves that $\rho$ given by \eqref{rho} is indeed a weak
solution to Cauchy problem \eqref{eq}, \eqref{eq-IC} and
\eqref{eq-BC}.

\subsection{Proof of Lemma \ref{uni-sol}}

Let us assume that $\ov\rho\in C^0([0,\delta];L^1(0,1)) \cap
L^{\infty}((0,\delta)\times (0,1))$ is a weak solution to  Cauchy
problem \eqref{eq}, \eqref{eq-IC} and \eqref{eq-BC}. Here $\delta$
is such that \eqref{delta} holds. Then by Lemma
\ref{lem-test-function}, for any fixed $t\in[0,\delta]$ and
$\varphi\in{C^1([0,t]\times[0,1])}$ with $\varphi(\tau,1)\equiv 0$
for $\tau\in [0,t]$,
 \begin{multline} \label{test}
  \int_0^{t} \int_0^1 \ov \rho(\tau,x)(\varphi_{\tau}(\tau,x)
   +\lambda(x, \ov W(\tau))\varphi_x(\tau,x))dxd\tau
   +\int_0^{t}u(\tau)\varphi(\tau,0)d\tau
   \\
 -\int_0^1 \ov \rho(t,x)\varphi(t,x)dx
    +\int_0^1\rho_0(x)\varphi(0,x)dx =0.
 \end{multline}
 where $\ov W(\tau):=\int_0^1\ov \rho(\tau,x)\, dx$.

 Let $\psi_0\in C_0^1(0,1)$, then we choose the test function
 as the solution to the following \emph{backward linear} Cauchy problem
 \beq
  \begin{cases}
   \psi_{\tau} +\lambda(x,\ov W(\tau))\psi_x =0, \quad & 0\leq \tau\leq t,\ 0\leq x\leq 1,\\
   \psi(t,x)=\psi_0(x), \quad & 0\leq x\leq 1,\\
    \psi(\tau,1)=0, \quad & 0\leq \tau\leq t.
   \end{cases}
  \eeq
For any fixed $t\in [0,\delta]$, we define the characteristic curves
$\ov \xi_1,\ov \xi_2,\ov \xi_3,\ov \xi_4$  by
\begin{align*}
  &\f{d\ov\xi_1}{ds}=\lambda(\ov\xi_1(s),\ov W(s)), \quad \ov\xi_1(t)=1,\\
  & \f{d\ov \xi_2}{ds}=\lambda(\ov\xi_2(s),\ov W(s)), \quad \ov\xi_2(0)=0,\\
   &\f{d\ov\xi_3}{ds}=\lambda(\ov\xi_3(s),\ov W(s)), \quad \ov\xi_3(t)=x, \quad \text{for}\  x \in
   [0,\ov\xi_2(t)],
    \\\label{2xi4}
   &\f{d\ov\xi_4}{ds}=\lambda(\ov\xi_4(s),\ov W(s)), \quad \ov\xi_4(t)=x, \quad \text{for}\  x \in
   [\ov\xi_2(t),1],
 \end{align*}
It is easy to see that there exist $\ov \alpha\in [0,t]$ and $\ov
\beta\in [0,1]$ such that
 \beq \ov\xi_3(\ov \alpha)=0 \quad \text{and } \quad \ov \xi_4(0)=\ov \beta.
 \eeq
In view of \eqref{test}, we compute
 \begin{align}\label{ov-rho-psi0}
     \int_0^1\ov \rho(t,x)\psi_0(x)\, dx
 & =\int_0^t u(\tau)\psi(\tau,0)\, d\tau + \int_0^1\rho_0(x)\psi(0,x)\, dx
   \nonumber\\
 &=\int_0^t u(\ov\alpha)\psi(\ov\alpha,0)\, d\ov \alpha
     + \int_0^1\rho_0(\ov \beta)\psi(0,\ov \beta)\, d\ov\beta
    \nonumber\\
 &=\int_0^t u(\ov \alpha)\psi_0(x)\, d\ov \alpha
   +\int_0^{1-\int_0^t\lambda(\ov \xi_1(\theta),\ov W(\theta))\, d\theta}
      \rho_0(\ov \beta)\psi_0(x)\, d\ov\beta
 \end{align}
By the definitions of $\ov \alpha,\ov \beta$ and similar to
\eqref{jacobi1}-\eqref{jacobi2}, we have
 \beq
  \f{d\ov \alpha}{dx}=\f{-1}{\lambda(0,\ov W(\ov\alpha))}e^{-\int_{\ov\alpha}^t {\lambda_x(\ov
   \xi_3(\theta),\ov W(\theta))}\, d\theta}
   \quad \text{and } \quad
 \f{d\ov \beta}{dx}=e^{-\int_0^t {\lambda_x(\ov
   \xi_4(\theta),\ov W(\theta))}\, d\theta}.
 \eeq
Thus, \eqref{ov-rho-psi0} is rewritten as
 \begin{align*}\label{equiv}
 &\quad \int_0^1\ov \rho(t,x)\psi_0(x)\, dx\nonumber\\
 &=\int_0^{\ov \xi_2(t)}\f{u(\ov \alpha)}{\lambda(0,\ov W(\ov \alpha))}
   e^{-\int_{\ov\alpha}^t \lambda_x(\ov \xi_3(\theta),\ov W(\theta))\, d\theta}\psi_0(x)\, dx
+\int_{\ov\xi_2(t)}^1\rho_0(\ov\beta)e^{-\int_0^t\lambda_x(\ov
 \xi_4(\theta),\ov W(\theta))\, d\theta}\psi_0(x)\, dx
 \end{align*}

Since $\psi_0\in C_0^1(0,1)$ and $t\in[0,\delta]$ are both
arbitrary, we obtain in $C^0([0,\delta];L^1(0,1)) \cap
L^{\infty}((0,\delta) \times (0,1))$ that
 \be \label{ov-rho}
 \ov\rho(t,x)=
   \begin{cases}
    \displaystyle \f{u(\ov\alpha)}{\lambda(0,\ov W(\ov\alpha))}
     \, e^{-\int_{\ov\alpha}^t \lambda_x(\ov\xi_3(\theta),\ov W(\theta)) d\theta},\quad
   & 0\leq x \leq \ov\xi_2(t) \leq 1,0\leq t\leq \delta,
   \\
   \rho_0(\ov\beta)\,  e^{-\int_0^t \lambda_x(\ov\xi_4(\theta),\ov W(\theta)) d\theta },\quad
   & 0\leq \ov\xi_2(t)\leq x \leq 1,0\leq t\leq \delta,
  \end{cases}
 \ee
which hence gives
 \begin{align*}
 \ov W(t)& =\int_0^1\ov\rho(t,x)\, dx
 \\& =\int_0^t u(\ov\alpha)d\ov\alpha
     +\int_0^{1-\int_0^t\lambda(\ov\xi_1(\theta),\ov W(\theta))d\theta} \rho_0(\ov\beta)d\ov\beta
    =F(\ov W)(t),   \quad  \forall t \in [0,\delta].
   \end{align*}
By \eqref{delta}, we claim that $\ov W\in\Omega_{\delta,M}$ and then
$\ov W \equiv W$ since $W$ is the unique fixed point of the map $W
\mapsto F(W)$ in $\Omega_{\delta,M}$. Consequently, we have
$\ov\xi_i\equiv \xi_i\ ( i=1,2,3,4)$ and $\ov
\alpha=\alpha,\ov\beta=\beta$. Finally, by comparing \eqref{rho-loc}
and \eqref{ov-rho}, we obtain $\ov\rho\equiv \rho$. This gives us
the uniqueness of the weak solution for small time.

\section*{Acknowledgements}
The authors would like to thank the professors Fr\'{e}d\'{e}rique
Cl\'{e}ment, Jean-Michel Coron for their interesting comments and
many valuable suggestions on this work.
%


\begin{figure}[htbp]
\begin{minipage}[b]{0.5\textwidth}\centering
\includegraphics[width=\textwidth]{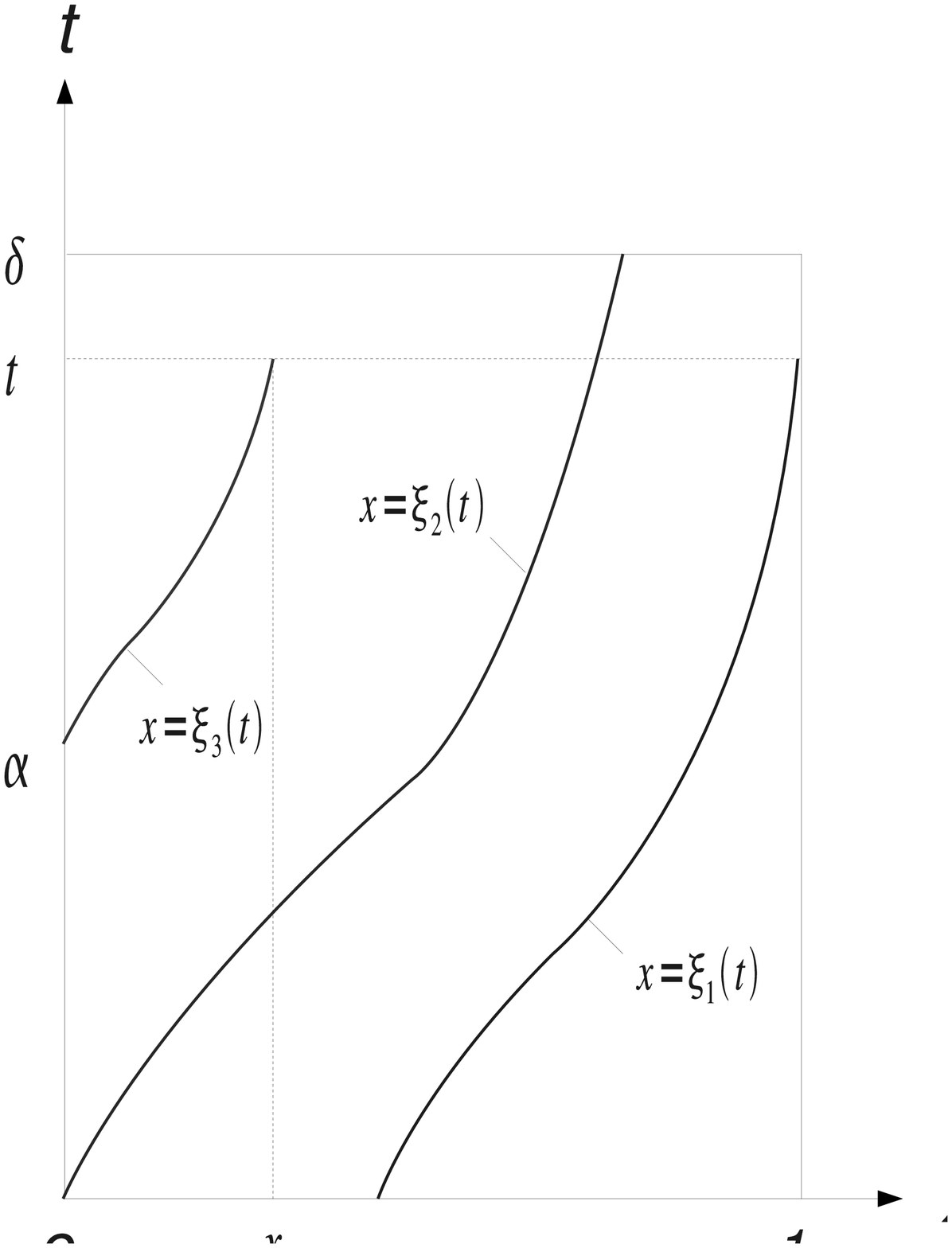}
\caption{Characteristics $\xi_1,\xi_2$ and $\xi_3$}  \label{Fig1}
\par\vspace{0pt}
\end{minipage}%
\hfill
\begin{minipage}[b]{0.5\textwidth}\centering
\includegraphics[width=\textwidth]{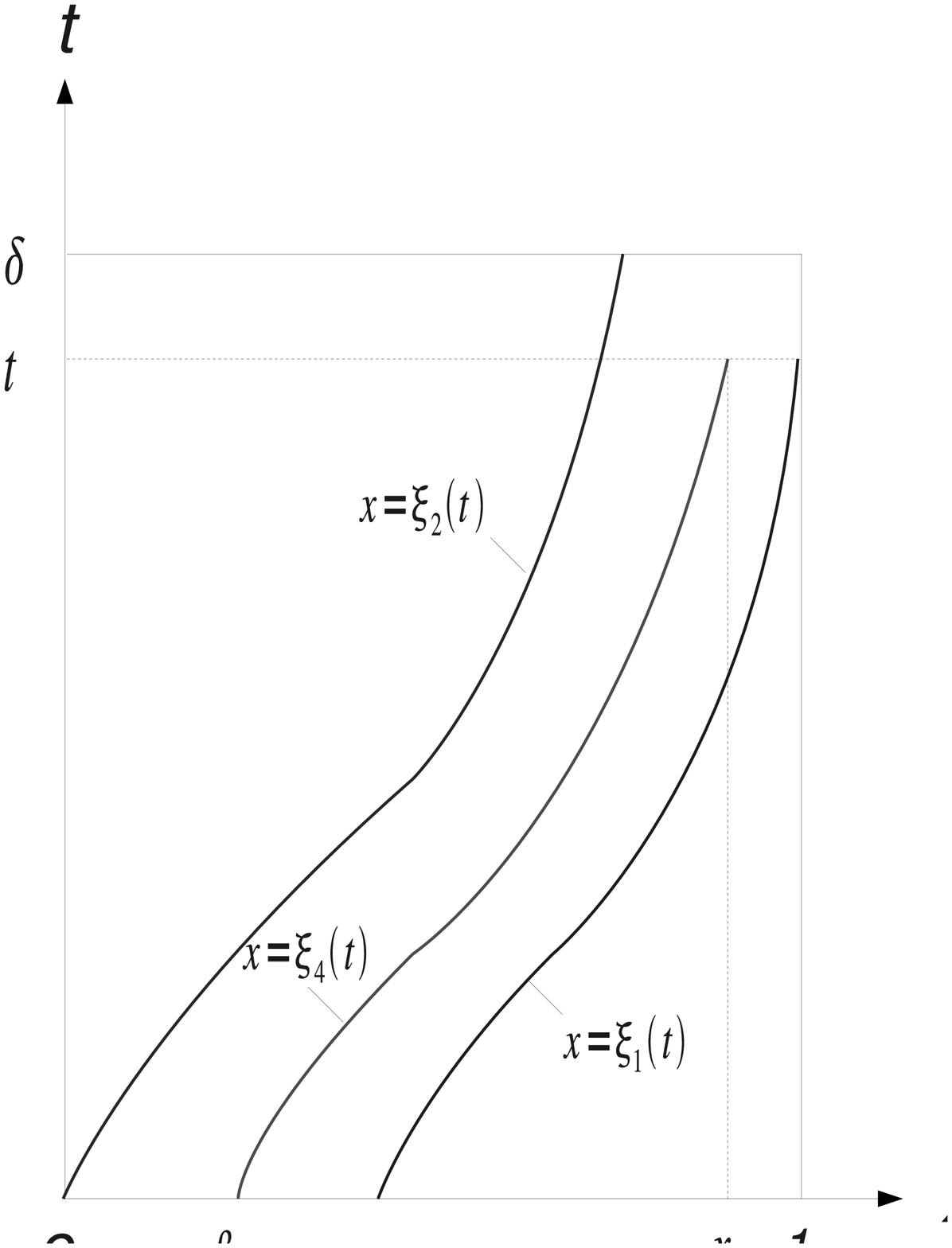}
\caption{Characteristics $\xi_1,\xi_2$ and $\xi_4$} \label{Fig2}
\par\vspace{0pt}
\end{minipage}%
\end{figure}

\vfill
\begin{figure}[htbp]
\begin{minipage}[b]{0.5\textwidth}\centering
\includegraphics[width=\textwidth]{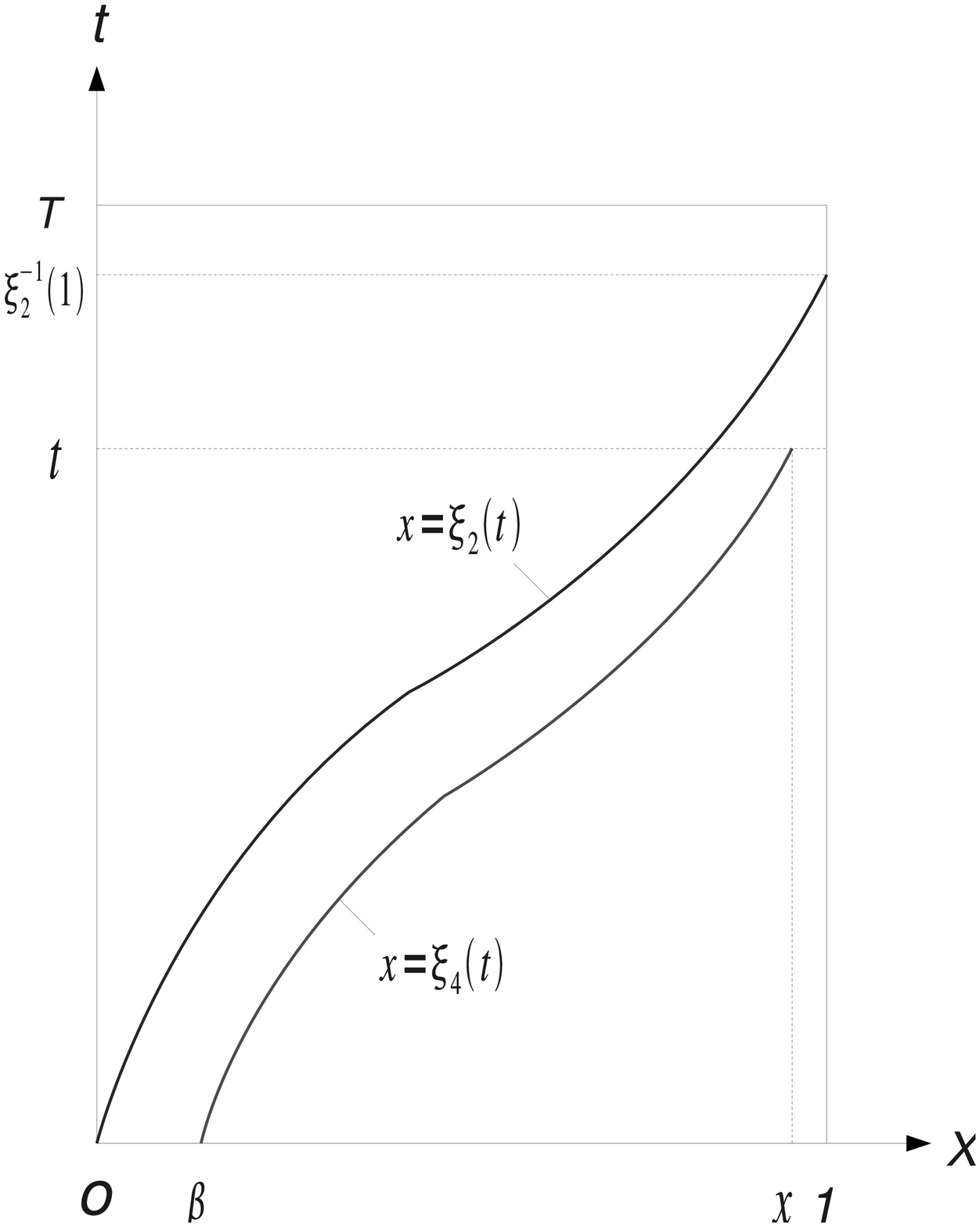}
\caption{Case $0\leq \xi_2(t)\leq x$, $0\leq t\leq \xi_2^{-1}(1)$}
\label{Fig3}
\par\vspace{0pt}
\end{minipage}%
\hfill
\begin{minipage}[b]{0.5\textwidth}\centering
\includegraphics[width=\textwidth]{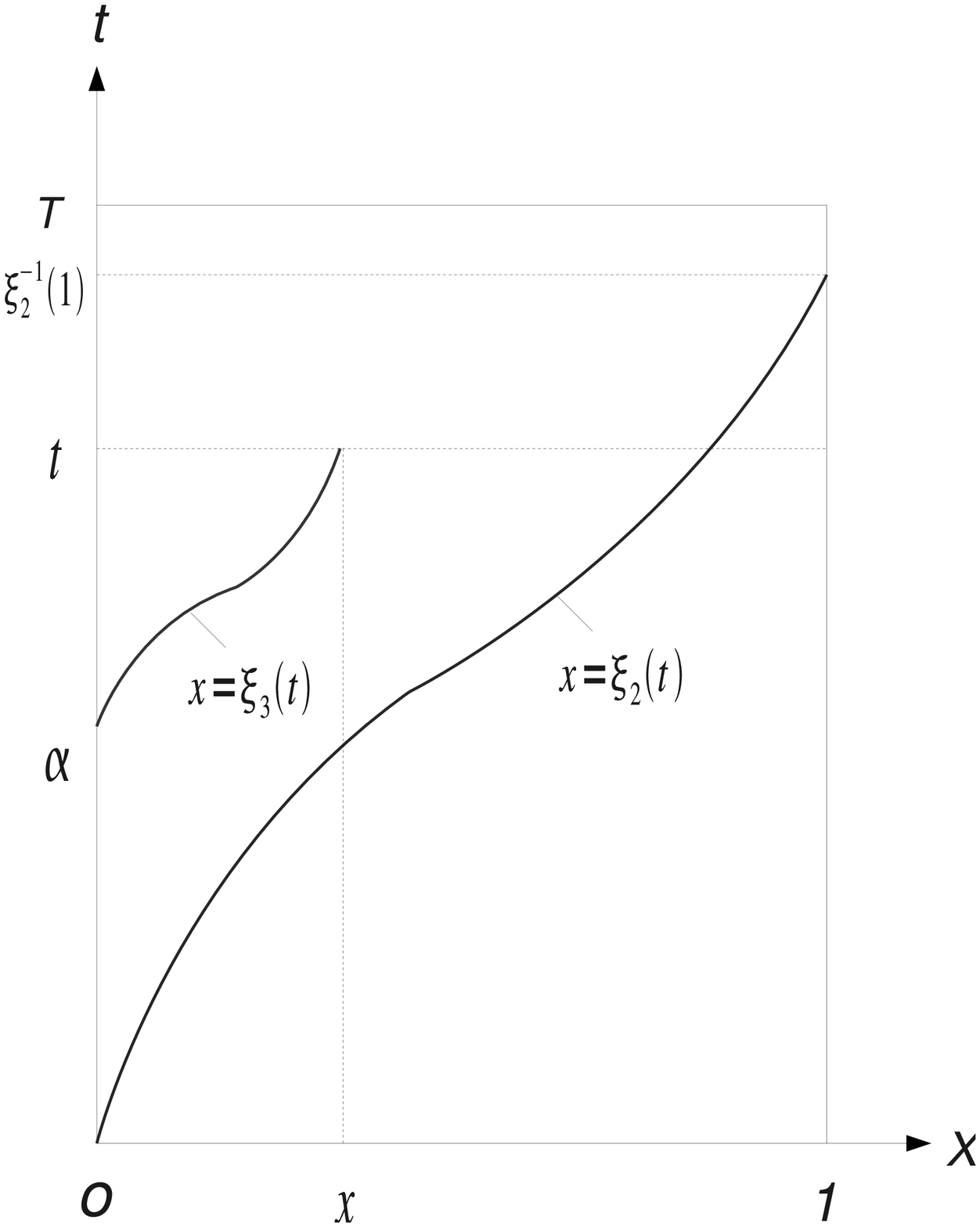}
\caption{Case $0\leq x\leq \xi_2(t)$, $0\leq t\leq \xi_2^{-1}(1)$}
\label{Fig4}
\par\vspace{0pt}
\end{minipage}%
\end{figure}

\vfill
\begin{figure}[htbp]
\begin{minipage}[b]{0.5\textwidth}\centering
\includegraphics[width=\textwidth]{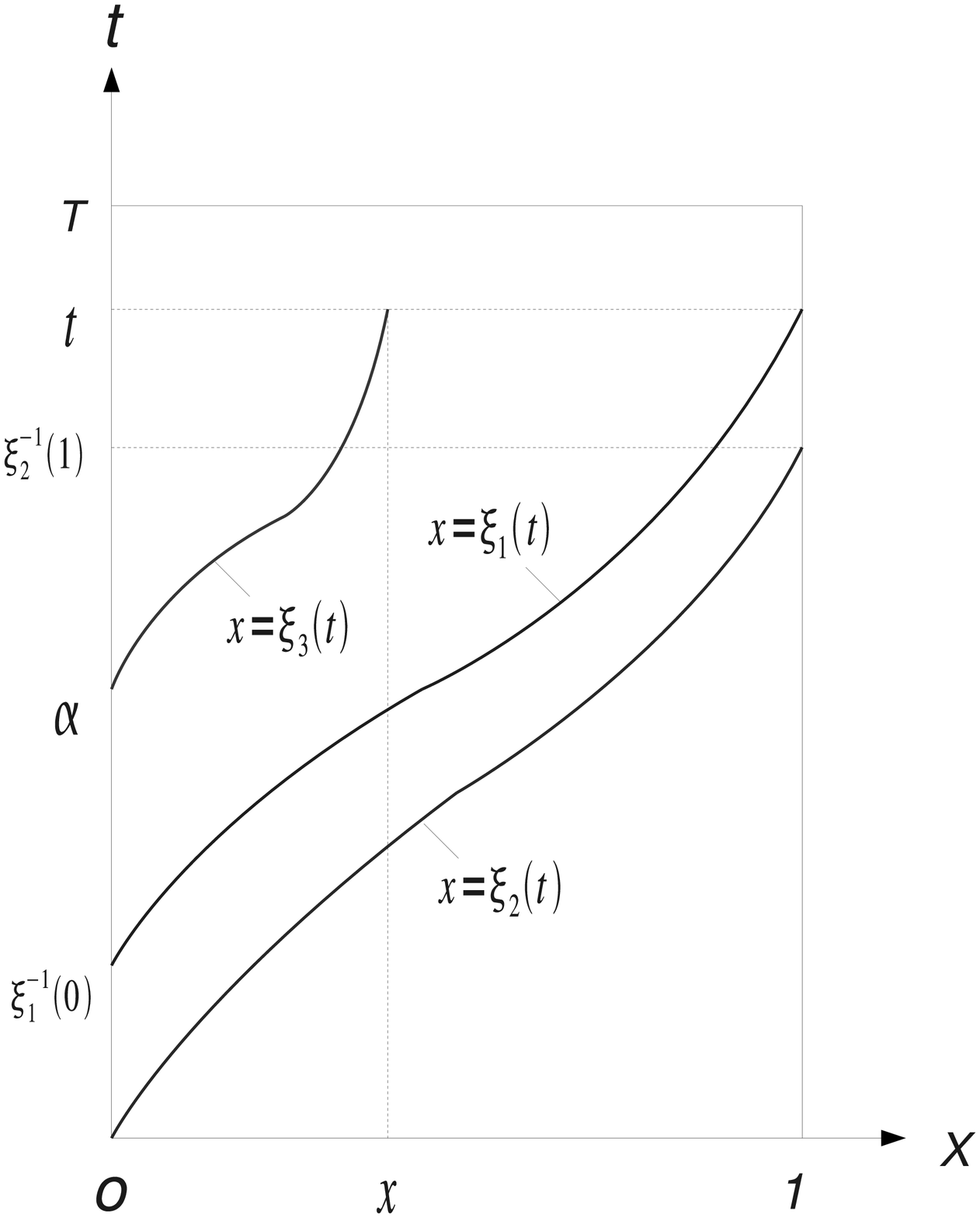}
\caption{Case $0\leq x\leq 1$, $\xi_2^{-1}(1)\leq t\leq T$}
\label{Fig5}
\par\vspace{0pt}
\end{minipage}%
\hfill
\begin{minipage}[b]{0.5\textwidth}\centering
\includegraphics[width=\textwidth]{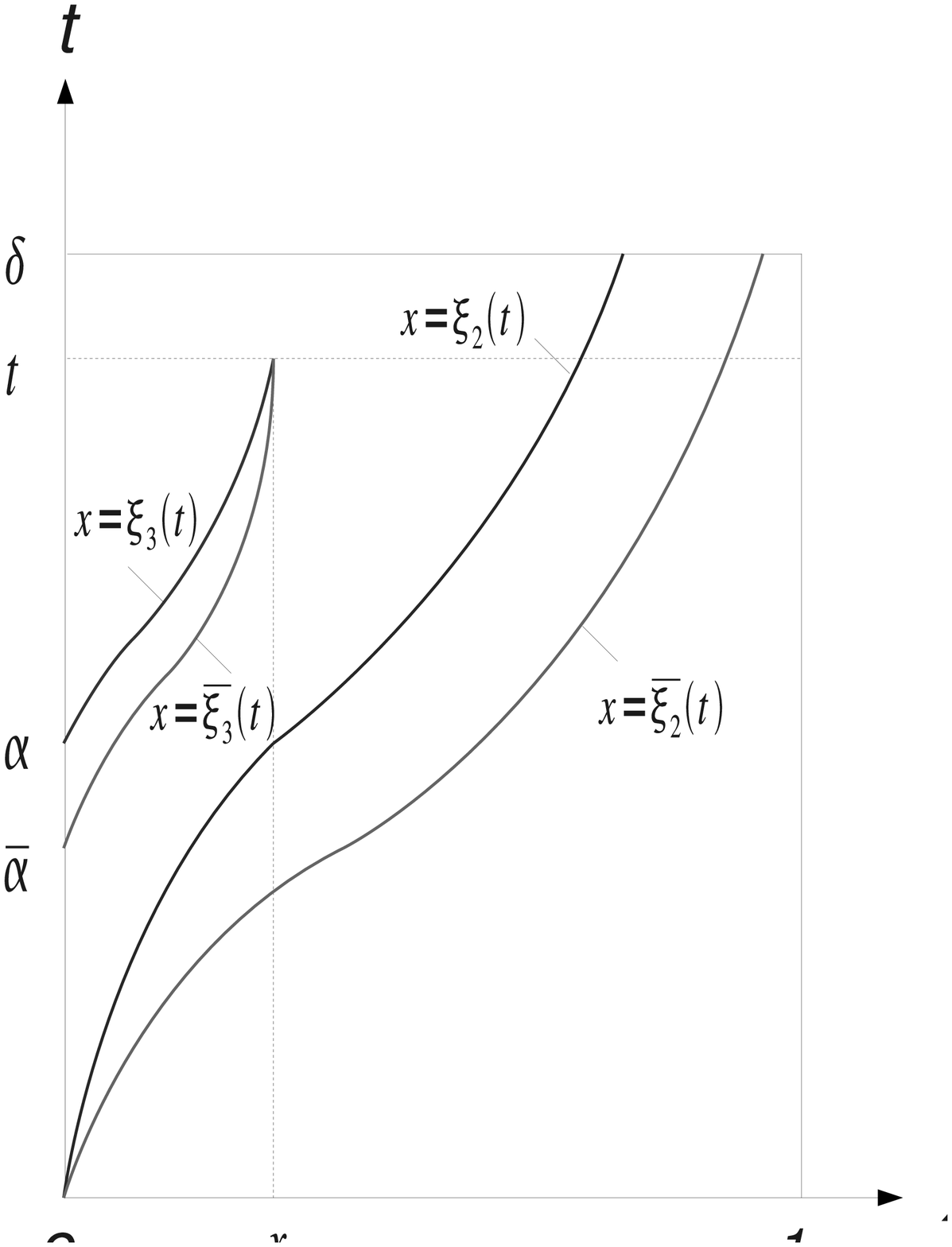}
\caption{Estimate for $x\in [0,\xi_2(t)]$} \label{Fig6}
\par\vspace{0pt}
\end{minipage}%
\end{figure}

\vfill
\begin{figure}[htbp]
\begin{minipage}[b]{0.5\textwidth}\centering
\includegraphics[width=\textwidth]{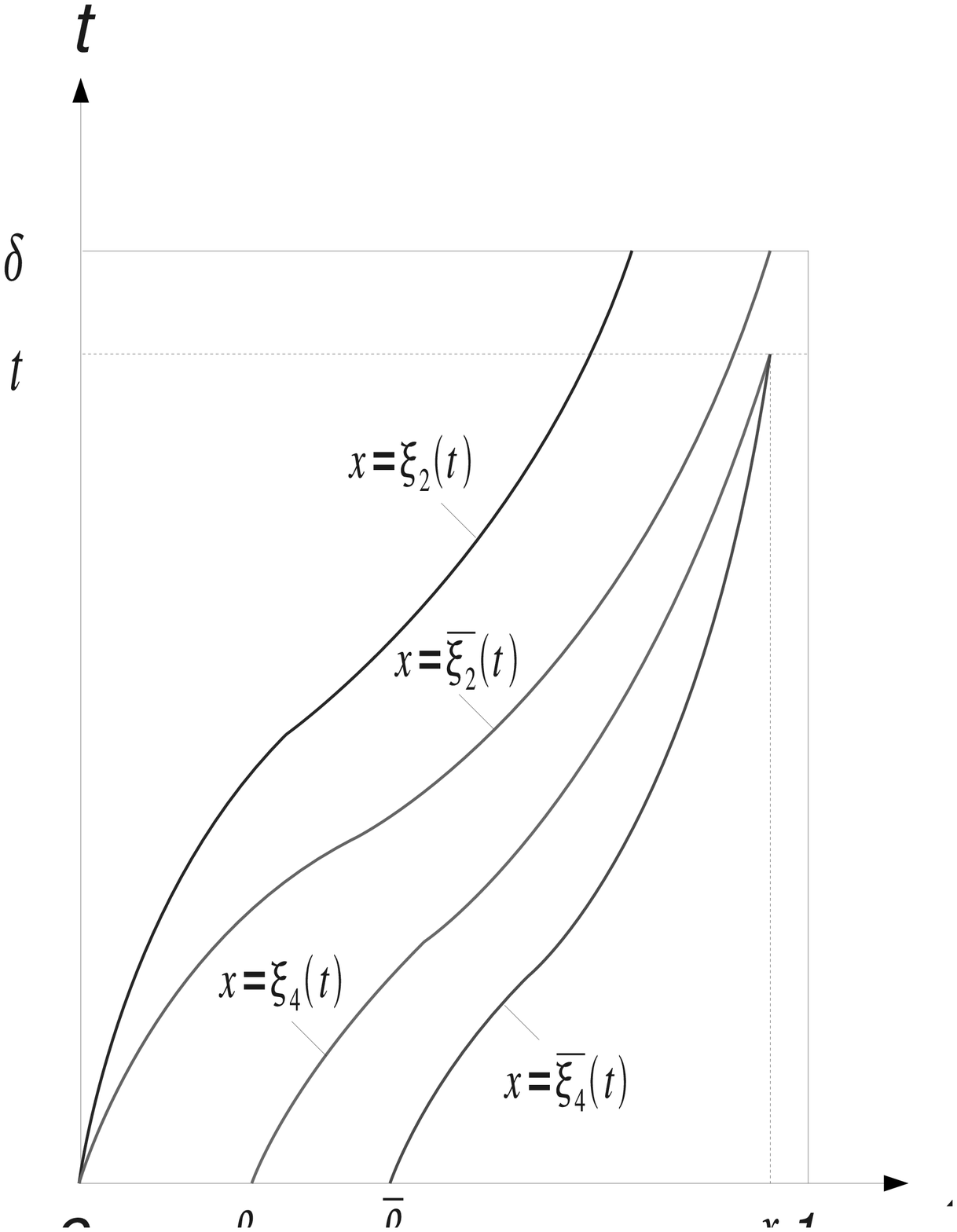}
\caption{Estimate for $x\in [\ov\xi_2(t),1]$} \label{Fig7}
\par\vspace{0pt}
\end{minipage}%
\hfill
\begin{minipage}[b]{0.5\textwidth}\centering
\includegraphics[width=\textwidth]{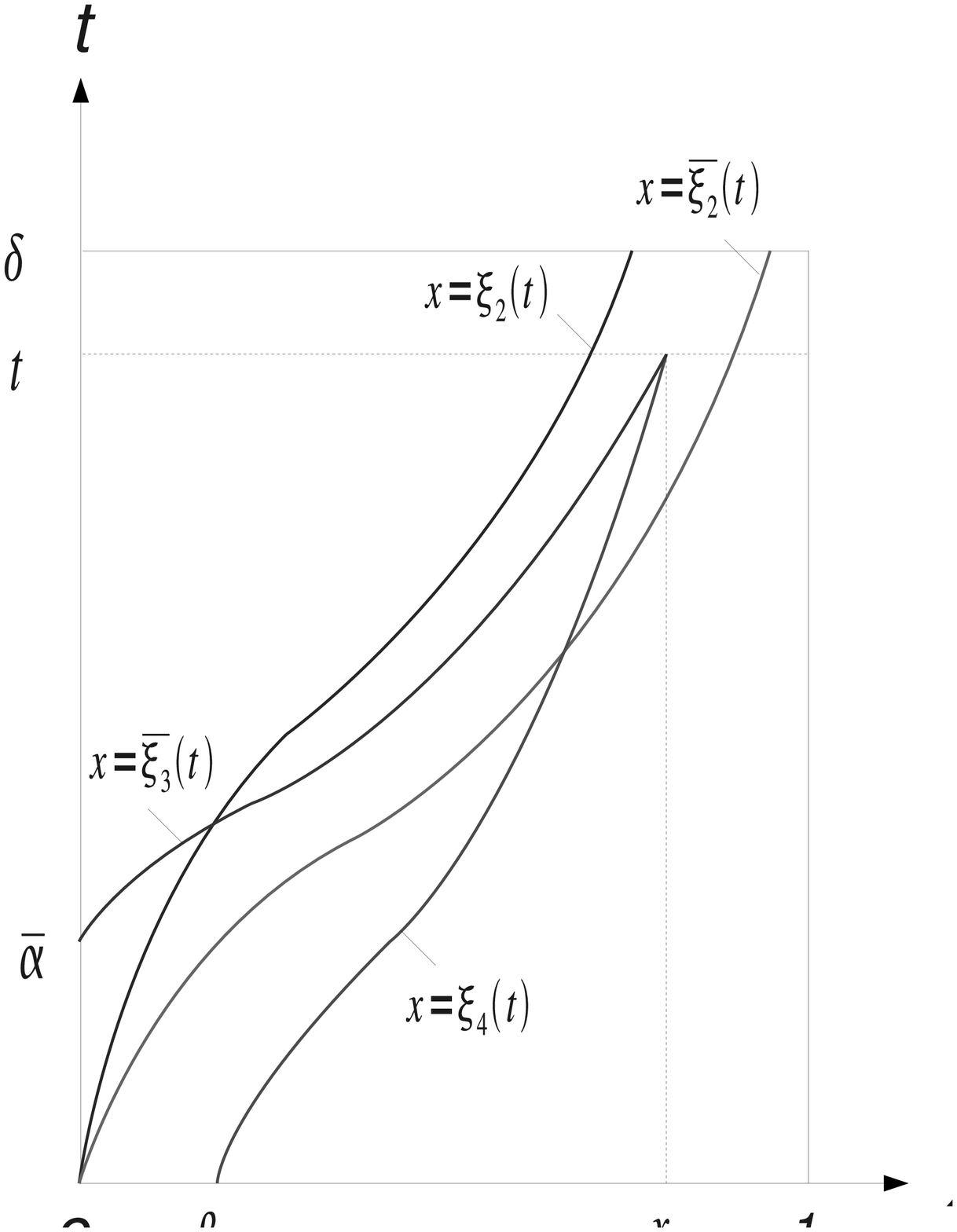}
\caption{Estimate for $x\in [\xi_2(t),\ov \xi_2(t)]$}  \label{Fig8}
\par\vspace{0pt}
\end{minipage}%
\end{figure}

\vfill
\begin{figure}[htbp]
\begin{minipage}[b]{0.5\textwidth}\centering
\includegraphics[width=\textwidth]{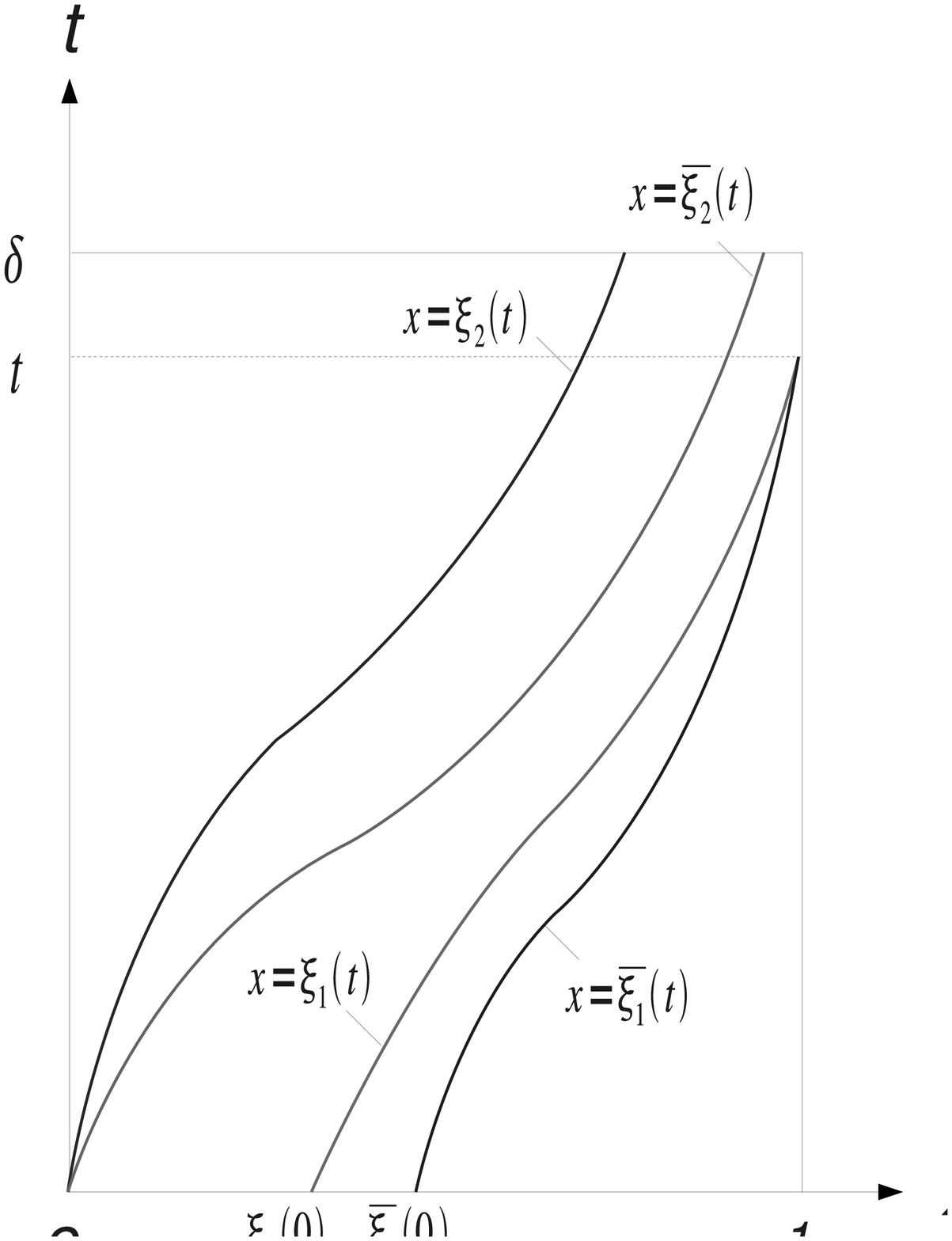}
\caption{Estimate of $|\ov y-y|$ for $t\in [0,\delta]$} \label{Fig9}
\par\vspace{0pt}
\end{minipage}%
\hfill
\begin{minipage}[b]{0.5\textwidth}\centering
\includegraphics[width=\textwidth]{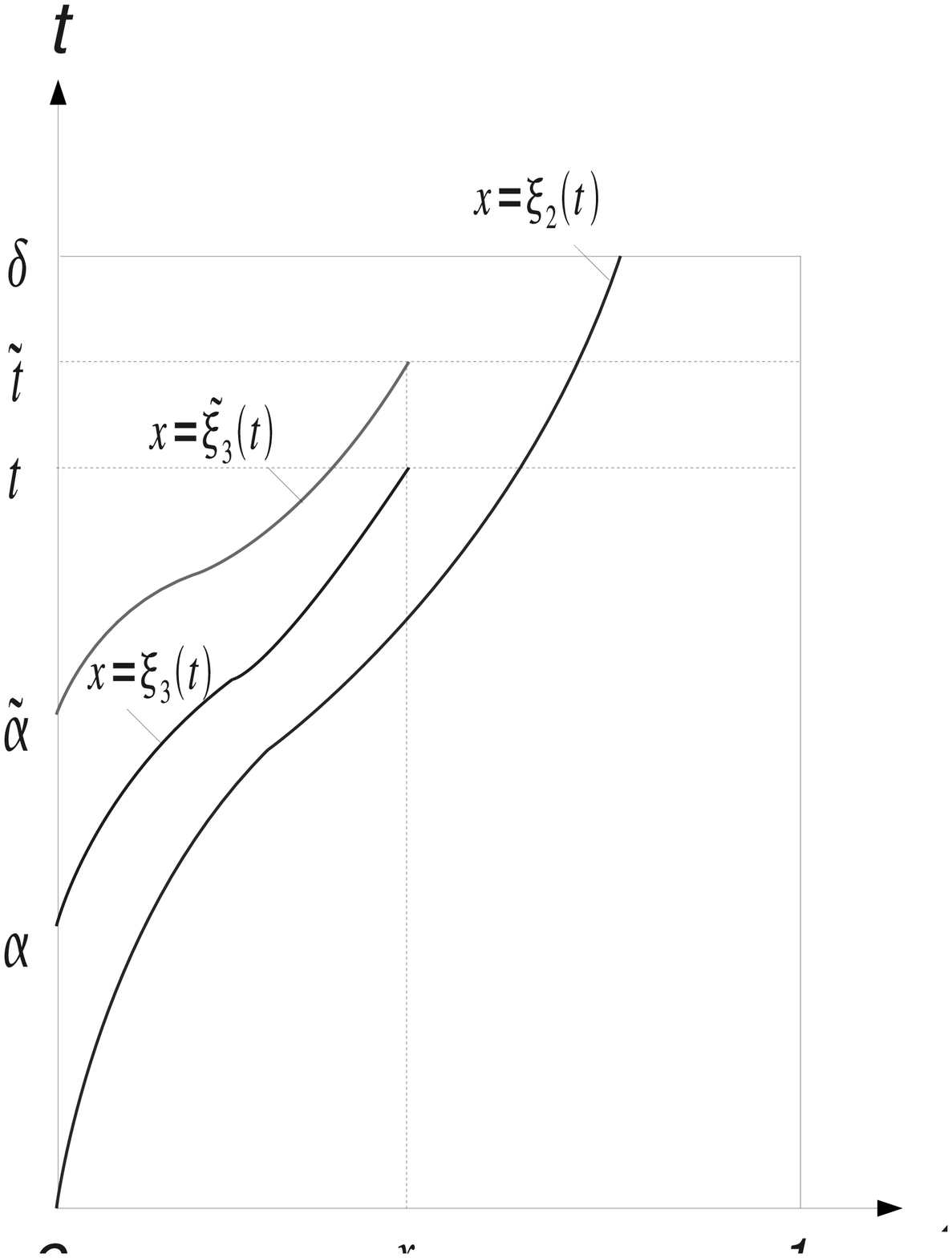}
\caption{Estimate for $x\in [0,\xi_2(t)]$} \label{Fig10}
\par\vspace{0pt}
\end{minipage}%
\end{figure}

\vfill
\begin{figure}[htbp]
\begin{minipage}[b]{0.5\textwidth}\centering
\includegraphics[width=\textwidth]{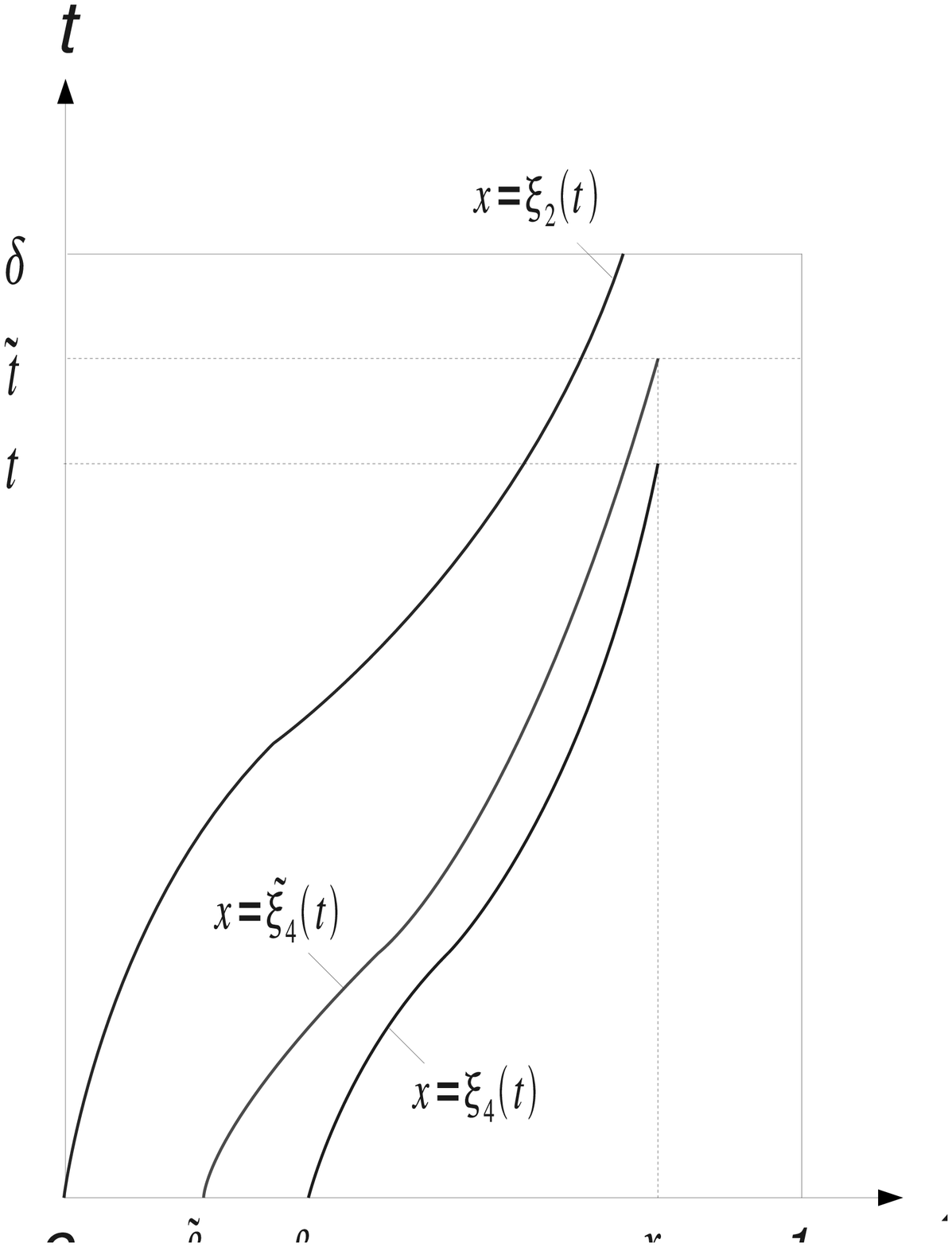}
\caption{Estimate for $x\in [\xi_2(\wt t),1]$} \label{Fig11}
\par\vspace{0pt}
\end{minipage}%
\hfill
\begin{minipage}[b]{0.5\textwidth}\centering
\includegraphics[width=\textwidth]{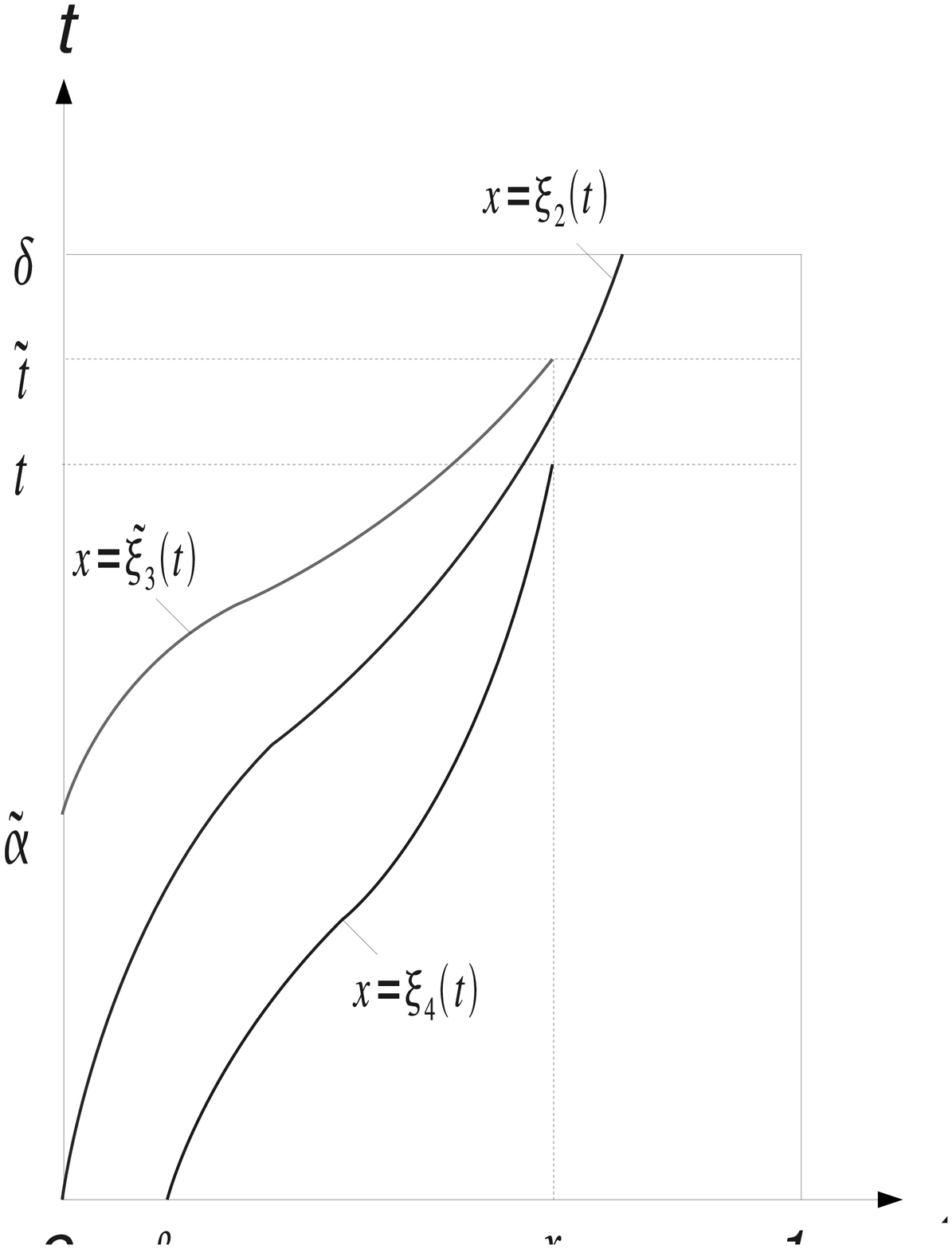}
\caption{Estimate for $x\in [\xi_2(t),\xi_2(\wt t)]$} \label{Fig12}
\par\vspace{0pt}
\end{minipage}%
\end{figure}

\end{document}